\documentclass{article}
\usepackage{amsfonts,amsmath,amssymb,amsthm}
\usepackage{graphicx} 
\usepackage{xcolor}
\usepackage{hyperref}
\usepackage{todonotes}
\usepackage{fullpage}
\usepackage{multicol}
\usepackage{ulem} 

\newtheorem{thm}{Theorem}[section]
\newtheorem{lemma}[thm]{Lemma}
\newtheorem{proposition}[thm]{Proposition}

\newcommand\Zf[1]{Z(#1)}
\newcommand\Zl[1]{Z_{(\ell)}(#1)}
\newcommand\Zone[1]{Z_{(1)}(#1)}
\newcommand\Ztwo[1]{Z_{(2)}(#1)}
\newcommand\Znum[2]{Z_{(#2)}(#1)}

\usetikzlibrary{graphs}
\usetikzlibrary{graphs.standard}
\usetikzlibrary{positioning, decorations.pathreplacing}
\usetikzlibrary{arrows}

\newcommand\plainell{*}
\newcommand\blockb{\Xi}
\title{Leaky Forcing: Extending Zero Forcing Results to a Fault-Tolerant Setting}

\author{Beth Bjorkman\thanks{Air Force Research Laboratory Sensors Directorate, Wright-Patterson Air Force Base, OH, USA (beth.morrison@us.af.mil)} \and Lei Cao\thanks{Department of Mathematics, Halmos College, Nova Southeastern University, Fort Lauderdale, FL, USA (lcao@nova.edu)} \and Franklin Kenter\thanks{Mathematics Department, United States Naval Academy, Annapolis, MD, USA (kenter@usna.edu)} \and  Ryan Moruzzi Jr\thanks{Department of Mathematics, California State University, Northridge, Northridge, CA, USA (ryan.moruzzi@csun.edu)} \and Carolyn Reinhart\thanks{Department of Mathematics and Statistics, Swarthmore College, Swarthmore, PA, USA (creinha1@swarthmore.edu)} \and Violeta Vasilevska\thanks{Department of Mathematics, Utah Valley University, Orem, UT, USA (violeta.vasilevska@uvu.edu)}
}

\begin{document}
\maketitle

\begin{abstract}
We study a recent variation of zero forcing called leaky forcing.  Zero forcing is a propagation process on a network whereby some nodes are initially blue with all others white. Blue vertices can ``force'' a white neighbor to become blue if all other  neighbors are blue. The goal is to find the minimum number of initially blue vertices to eventually force all vertices blue after exhaustively applying the forcing rule above.
Leaky forcing is a fault-tolerant variation of zero forcing where certain vertices (not necessarily initially blue) cannot force. The goal in this context is to find the minimum number of initially blue vertices needed that can eventually force all vertices to be blue, {\it regardless} of which small number of vertices can't force. This work extends results from zero forcing in terms of leaky forcing. In particular, we provide a complete determination of leaky forcing numbers for all unicyclic graphs and upper bounds for generalized Petersen graphs. We also provide bounds for the effect of both edge removal and vertex removal on the $\ell$-leaky forcing number. Finally, we completely characterize connected graphs that have the minimum and maximum possible $1$-leaky forcing number (i.e., when $\Zone{G} = 2$ and when $\Zone{G} = |V(G)|-1$).
\end{abstract}

\footnotetext{~\\The views
expressed in this article are those of the authors and do not reflect the official policy or position of the Air Force Research Laboratory, the U.S. Naval Academy, the Department of the Air Force, the Department of the Navy, the
Department of Defense, or the U.S. Government.}

\section{Introduction}

Zero forcing is a graph coloring process during which an initial set of blue vertices force the remaining vertices to become blue via repeated applications of a color change rule. It was originally introduced in \cite{AIM} as an upper bound for the maximum nullity over all matrices whose zero/nonzero pattern is dictated by a graph. Since then, it has grown into a parameter of independent interest and many variants have been introduced. 

The standard \textit{color change rule} states that a blue vertex $b$ may \textit{force} a white vertex $w$ to become blue if $w$ is the only white neighbor of $b$. Any set of initially blue vertices of a graph which results in the entire graph being colored blue after repeated applications of the color change rule is called a \textit{zero forcing set} for that graph. The size of a smallest zero forcing set for the graph is called the \textit{zero forcing number} of that graph, denoted $\Zf{G}$. 

Applications of zero forcing include quantum controllability \cite{zfquantum}, leader-follower dynamics \cite{zfleader}, and sensor allocation on electrical grids \cite{pmu,powerdominiation}. For each of these applications, a particular set of agents or nodes must be identified before the full system is realized. This is done such that regardless of how the system is realized, the system can be observed or controlled using only information known at the chosen agents or nodes. In a linear-algebraic context, zero forcing identifies a set of columns (vertices) that covers a set of all nonpivot positions for any real symmetric matrix representing a graph \cite{kenterlin}. Hence, a zero forcing set can be used as a foundation to backsolve {\it any} linear system involving a matrix whose zero/nonzero pattern is given by that graph.

We continue to study the variation of zero forcing called \textit{$\ell$-leaky forcing} (or \textit{leaky forcing} for short), which was introduced in \cite{og}. In this variant, after an initial set of blue vertices are chosen, $\ell$ vertices of the graph are adversarially chosen to be leaks. The \textit{leaky color change rule} states that a blue vertex $b$, which is not a leak, may force a white vertex $w$ to become blue if $w$ is the only white neighbor of $b$. Note that vertices chosen as leaks can never force. A set of initially blue vertices of a graph which results in the entire graph being colored blue after repeated applications of the leaky color change rule, regardless of the location of the $\ell$ leaks, is called an $\ell$-\textit{leaky forcing set} for that graph. The size of a smallest $\ell$-leaky forcing set for the graph is called the $\ell$-\textit{leaky forcing number} of that graph, denoted $\Zl{G}$. Note that $\Znum{G}{0} = \Zf{G}$.

 The main challenge of leaky forcing is that to certify a set is an $\ell$-leaky forcing set, one must verify it overcomes {\it all} possible combination of $\ell$ leaks in the graph. Alameda, Kritschgau, and Young \cite{resl} extend several known results for the classical zero forcing to leaky forcing including results on edge- and vertex- removal, forcing chains, and path graphs. Abbas \cite{abbas2023resilient} applies leaky forcing to the context of controllability of networks. Herrman \cite{herrman2022d} gave exact values for the leaky forcing number for select families of graphs, including the hypercube and a special case of the generalized Petersen graph.

The work presented in this paper significantly extends several concepts known for zero forcing as well as several of the results for leaky forcing mentioned above. After covering the preliminary materials (Section \ref{sec:pre}), we accomplish the following.
\begin{itemize}
    \item We give a complete classification for the $\ell$-leaky forcing numbers of all unicyclic graphs. This includes separate cases for $\ell=1$ (Theorems \ref{thm:3Cycle} and \ref{thm:4+cycle}), $\ell=2$ (Theorem \ref{thm:uni2}), and $\ell \ge 3$ (Theorem \ref{thm:uni3+}). This extends the work by Fallat and Hogben \cite{fallat2007minimum} as well as Row \cite{row2012technique} for zero forcing on unicyclic graphs.
    \item We determine the $1$-leaky forcing number for certain cases of the generalized Petersen graph (Theorems \ref{thm:gponePn2exact} and \ref{thm:gponePn3exact}). Further, we prove an upper bound for the 1-leaky and 2-leaky forcing for all generalized Petersen graphs (Theorems \ref{thm:gpone} and \ref{thm:gptwo}, respectively). These bounds extend previous work by Hermann \cite{herrman2022d} and Rashidi and Poursalavati \cite{rashidi2020computing}.
 \item We provide an extension of edge- and vertex-removal results (Theorems \ref{thm:edgeremove} and \ref{thm:vtxremove}, respectively), generalizing the results for 1-leaky forcing of Alameda, Kritschgau, Warnberg, and Young \cite{resl}. 
\item We give a complete classification for connected graphs that achieve the extremal values for $\Zone{G}$, specifically when $\Zone{G} = 2$ (Theorem \ref{thm:Zoneclass2}) and when $\Zone{G} = n-1$ (Theorem \ref{thm:Zoneclass}).
\end{itemize}

\section{Preliminaries} \label{sec:pre}

    We begin by introducing the notation, terminology, and previous results that will serve as a foundation for what follows. We will refer to graphs as simple undirected graphs without loops. A graph $G$ is formally $G = (V,E)$ where $V$ is the set of vertices of $G$ and $E \subseteq {V \choose 2}$ (i.e., unordered pairs of elements of $V$) is the set of edges of $G$. If the graph is not necessarily clear we write $V(G)$ and $E(G)$. Two vertices $u$ and $v$ contained in the same edge are said to be ${\it adjacent}$ or ${\it neighbors}$; an edge between $u$ and $v$ is written $uv$. The {\it degree} of a vertex $v$ in $G$, denoted $\deg(v)$, is the number of neighbors of $v$.
    Given a vertex $v \in G$, the graph $G - v$ has vertex set $V \setminus \{v\}$ and edge set $\{e \in E \colon v \not \in e\}$. Likewise, for a subgraph $H$, $G-H$ has vertex set $V(G)\setminus V(H)$ and edge set $\{e \in E \colon e \subseteq V(G)\setminus V(H)\}$. Given an edge $e \in G$, the graph $G-e$ has vertex set $V$ and edge set $E \setminus \{ e \}$. For a graph $G$, a subgraph $H$, and a vertex $v$, the subgraph $H + v$ is the induced subgraph of $G$ on $V(H) \cup \{v\}$.

    We will discuss several families of graphs throughout the paper. Let $K_n$ be the \textit{complete graph} on $n$ vertices, $C_n$ be the \textit{cycle graph} on $n$ vertices, let $P_n$ be the \textit{path graph} on $n$ vertices, and let $S_n$ be the \textit{star graph} on $n$ vertices. A \textit{tree} is a connected graph which contains no cycles as a subgraph and a \textit{forest} is a disjoint union of trees. A \textit{unicyclic graph} is a connected graph which contains exactly one cycle; structural aspects of unicyclic graphs will be defined more carefully in Section \ref{sec:Unicyclic}. The \textit{generalized Petersen graph} denoted $P(n,k)$ will be carefully defined in Section \ref{sec:gp}. For other graph theory terminology or notation see \cite{West}.

    The leaky forcing number is known for several basic graphs, which we recall now and refer to throughout the paper.

\begin{proposition}[\cite{og}, Propositions 3.1, 3.4, and 3.5]\label{prop:KP}\label{prop:basicgraphresults} The following $\ell$-leaky forcing numbers are known:
\begin{enumerate}
\item For the path graph $P_n$ on $n\geq 2$ vertices, $\Zf{P_n}=1$, $\Zone{P_n}=2$, and if $\ell\ge 2$, then $\Zl{P_n}=n$.
\item  For the complete graph $K_n$ on $n\geq 2$ vertices, if $0\leq \ell\leq n-1$, then $\Zl{K_n}=n-1$ and $\Znum{K_n}{n}=n$.
\item For the cycle graph $C_n$ on $n\ge 3$ vertices, if $\ell=0,1$, then $\Zl{C_n}=2$ and if $\ell\geq 2$, then $\Zl{C_n}=n$.
\end{enumerate}
\end{proposition}

We also make frequent use of following known results about $\ell$-leaky forcing: the first concerning vertices of low degree, the second establishing a chain of inequalities for leaky forcing numbers as the number of leaks increases, and the third giving a criteria for a subset of vertices to be an $\ell$-leaky set of a graph $G$.

\begin{lemma}[\cite{og}, Lemma 2.4] \label{lem:lowdegree}
Let $G$ be a graph and choose $\ell \ge 0$. Then, every $\ell$-leaky forcing set contains every vertex of degree $\ell$ or less.
\end{lemma}

\begin{lemma}[\cite{og}, Lemma 2.1]
\label{lem:moreleaksinq}
For any graph $G$ on $n$ vertices,
\[  \Zf{G} \le \Zone{G} \le \Ztwo{G} \le \cdots \le \Znum{G}{n}. \]
\end{lemma}

\begin{lemma}[\cite{resl}, Theorem 2.5] \label{lem:twodifferentleaky}
A set $S$ is an $\ell$-leaky forcing set of $G$ if and only if it is an $(\ell-1)$-leaky forcing set such that for all vertices $v \in V(G) \setminus S$, there are two different sequences of forces, one where $u$ forces $v$ and one where $u' \ne u$ forces $v$.
\end{lemma}

In addition to the previous results, we make use of the notion of forts.
The concept of a fort was introduced in \cite{FortsRef}, and extended to $\ell$-leaky forcing in \cite{og}. In a graph $G$, an $\ell$-{\it leaky fort} 
is a nonempty set of vertices $F \subseteq V(G)$ such that all but $\ell$ vertices in $V \setminus F$ are not adjacent to exactly one vertex in $F$. In other words, a fort is a nonempty set of white vertices for which, when all other vertices are blue, there is a configuration of leaks which prevents any of the fort vertices from being forced. Forts provide a convenient argument for a lower bound for the $\ell$-leaky forcing number in the following way. 

\begin{proposition}[\cite{og}, Proposition 2.3]\label{prop:LeakyIFFforts}
A set $S \subseteq V(G)$ is an $\ell$-leaky forcing set of $G$ if and only if $S$ intersects all $\ell$-leaky forts.    
\end{proposition}

\section{Unicyclic Graphs}\label{sec:Unicyclic}\label{sec:unicyclic}
In this section, we characterize the $\ell$-leaky forcing number, $\Zl{G}$, for all unicyclic graphs, $G$, and all $\ell\ge 1$.  A \textit{unicyclic graph} is a connected graph which contains exactly one cycle. We show that unicyclic graphs have nearly the same minimum $1$-leaky and $2$-leaky forcing sets as trees. That is, the minimum $1$-leaky (respectively $2$-leaky) forcing set is the set of all vertices of degree at most 1 (respectively 2) but with up to two additional vertices. When $\ell \ge 3$, we show that unicyclic graphs have the same minimum $\ell$-leaky forcing sets as that of trees (see \cite{resl}). 

For classical zero forcing, Row proved the following.
\begin{thm}[\cite{row2012technique}, Theorem 4.6]
    For a unicyclic graph, $G$, $\Zf{G} = P(G)$ where $P(G)$ is the induced path cover number (as defined in \cite{AIM}).
\end{thm}
Unlike this previous result for zero forcing, we will characterize the $\ell$-leaky forcing numbers based solely on the characteristics of the unicyclic graph without resorting to another graph parameter.

We now introduce particular notation and definitions we will use in the context of unicyclic graphs. Let $G$ be a unicyclic graph. The \textit{girth} of $G$, $g$, is the length of its only (and shortest) cycle, $C$. 
We will label the vertices of $C$ as $\{c_1,c_2,c_3, \ldots c_g\}$ and call these vertices \textit{cycle vertices}. For any cycle vertex $c_i$, we will denote the adjacent cycle vertices to be $c_{i-1}$ and $c_{i+1}$;
where applicable, the indices of the cycle vertices are taken to be modulo $g$ (e.g., $c_1$ is adjacent to $c_g$).
The vertices $c_{i-1}$ and $c_{i+1}$ will be called the \textit{cycle neighbors} of $c_i$. 

Note the graph $G-C$ (e.g., the graph resulting from removing all cycle vertices and the edges incident to them from $G$) is a forest. We will denote each tree in the forest as $T_{c_i,j}$ where $c_i$ is the unique cycle vertex incident to the tree in $G$ and $j$ is an index provided arbitrarily to distinguish the trees with the same incident cycle vertex. Note that the subgraph $T_{c_i,j}$ excludes the cycle vertex $c_i$. For a given unicyclic graph, $\mathcal{I}$ denotes the set of all possible tree indices $(i,j)$ for which there is a tree $T_{c_i,j}$.

Note that unicyclic graphs can be thought of as a tree with the addition of one extra edge. The following result is a known characterization for the $\ell$-leaky forcing set for trees. 

\begin{thm}[\cite{resl}, Theorem 4.4] \label{thm:LeakyTree}
    If $T$ is a tree and $~\mathcal{U}$ is the set of vertices in $T$ with degree at most $\ell$ where $\ell \ge 1$, then $\Zl{T} = \lvert \mathcal{U} \rvert$.
\end{thm} 
Hence, by Lemma \ref{lem:lowdegree}, for $\ell \ge 1$, the minimum $\ell$-leaky forcing set is precisely the set $\mathcal{U}$ consisting of vertices with degree $\ell$ or less.

We now prove three lemmas which will be useful for the proofs of our characterizations for leaky forcing on unicyclic graphs.

\begin{lemma}  \label{lem:forcetree}
Let $\ell \ge 1$ and let $G$ be a unicyclic graph with cycle subgraph $C = \{c_1, \ldots, c_g\}$. Let $I \subseteq \mathcal{I}$ be a subset of tree indices with $|I| \le \ell$ and let $S$ be a set of vertices such that both of the following are true:
\begin{enumerate}
    \item  All vertices of degree $\ell$ or less of $G$ are in $S$ and 
    \item All vertices that are not in $S$ are in $\bigcup_{(i,j) \in I} T_{c_{i}, j}$.
\end{enumerate}  
Then, $S$ is an $\ell$-leaky forcing set of $G$.
\end{lemma}

\begin{proof}
    For each $i$ such that $(i,j)\in I$, let $j_1,\dots,j_{k_i}$ be the set of indices such that $(i,j_1),\dots,(i,j_{k_i})\in I$. Apply Theorem \ref{thm:LeakyTree} individually to each of the subgraphs $G \left[T_{c_i,j_1}\cup\dots\cup T_{c_i,j_{k_i}} \cup \{c_i\} \right]$ for each $i$.
\end{proof}

 \begin{lemma} \label{lem:specificzftree}
 Let $T$ be a tree with at least $2$ vertices and choose a leaf, $v^\plainell$. Let $L$ be the set of leaves. Then, $L \setminus \{v^\plainell\}$ is a zero forcing set of $T$.
 \end{lemma}

\begin{proof}
Let $v^\plainell$ be the root of $T$. We proceed by induction on the height of the tree, $h$, measured from the root $v^\plainell$. For $h=1$, $T = S_n$, and the result follows. For the induction step, observe that coloring all leaves blue except $v^\plainell$ will allow all other leaves to force their parents in the tree. After this forcing step, observe that the set of blue vertices with at least one white neighbor is a forest. In each component, all but at most one leaf is blue and the height of the component is at most $h-1$; hence, applying the induction hypothesis completes the proof.
\end{proof}

\begin{lemma}  \label{lem:unitree}
Let $G$ be a unicyclic graph with cycle subgraph $C$. Then, coloring all of the leaves of $G$ blue is sufficient to force each subtree $T_{c_i,j}$ that does not contain a leak. Furthermore, each $T_{c_i,j}$ forced in this way will also result in $c_i$ being forced. 
\end{lemma}

\begin{proof}
Color all of the leaves of $G$. Consider the tree $T_{c_i,j} + c_i$.
Necessarily, this means all of the leaves of $T_{c_i,j} + c_i$ will be initially colored, except for $c_i$. We now apply the proof of Lemma \ref{lem:specificzftree} (using $c_i$ as the root) without considering any forces from any other $T_{c_i,j}$. This necessarily guarantees that $T_{c_i,j} + c_i$ will be forced with $c_i$ forced last. 
\end{proof}

We are now ready to characterize the $1$-leaky forcing number for unicyclic graphs. First, we will consider the case with girth $3$.

\begin{thm}\label{thm:3Cycle} 
    Let $G$ be a unicyclic graph with girth $3$ and $Q = \{v\in V(G) : \deg(v)= 1 \}$ such that $|Q|=q$, then 
    \[\Zone{G} = \begin{cases}
    q  & \text{if there are at least two cycle vertices of degree at least $4$ or}\\
    &\text{\quad all cycle vertices have degree at least $3$,} \\
     q +2 & \text{if there is at most one cycle vertex of degree $3$ and at least two cycle vertices of degree $2$,} \\ 
     q +1 & \text{otherwise.}\\
     \end{cases}
\]
\end{thm}

\begin{proof}

We will denote the cycle vertices $c_u$, $c_v$, and $c_w$. Without loss of generality we will label the vertices so that $\deg(c_u) \ge \deg(c_v) \ge \deg(c_w) \ge 2$. Also, note that by Lemma \ref{lem:lowdegree}, every $1$-leaky forcing set $S$ of $G$ must contain all the leaves of $G$, i.e., $Q\subseteq S$.

{\bf Case 0:}  $\deg(c_u) = \deg(c_v) = \deg(c_w) = 2$.

In this case, $G\cong C_3$ and $\Zone{C_3} = 2$ by Proposition \ref{prop:basicgraphresults}. Then, since $q=0$, we have $\Zone{G} = q+2$.

For the remainder of the argument, we assume that at least one cycle vertex has degree at least 3 (e.g., $\deg(c_u) \ge 3$).

\textbf{Case 1:} $\deg(c_u) \ge \deg(c_v) \ge \deg(c_w) \ge 3$. 

We will show that $S=Q$ is a 1-leaky forcing set. If the leak is on $c_u$, $c_v$, or $c_w$, apply Lemma \ref{lem:unitree} to force each tree $T_{c_i,j}+c_i$ for all $i \in \{u,v,w\}$ and all applicable $j$. Since every cycle vertex has a tree, we are done.

    Now suppose that the leak is within $T_{c_i,j}$ for some $i \in \{u,v,w\}$. The other vertices of the cycle will necessarily be forced, and $c_i$ will be forced by one of its cycle neighbors. Then, apply Lemma \ref{lem:forcetree} to force $T_{c_i,j}$ (and hence all of $G$). Thus, $\Zone{G} = q$. 
    
    \textbf{Case 2:} $\deg(c_u) \ge \deg(c_v) \ge 4$ and $\deg(c_w) = 2$. 

    We will show that $S=Q$ is a 1-leaky forcing set. If the leak is on some $c_x$ for $x\in\{u,v,w\}$, apply Lemma \ref{lem:unitree} to force $T_{c_i,j}$ and $c_i$ for $i \in \{u,v\}$ and all applicable $j$. There is a cycle vertex $c_y \ne c_x$ with $\deg(c_y) \ge 4$ which can force $c_w$ and we are done.

If the leak is in $T_{c_i,j}$ for $i \in \{u,v\}$, apply Lemma \ref{lem:forcetree} to force $T_{c_i,j}$ (and hence all of $G$). Thus, $\Zone{G} = q$. 

     \textbf{Case 3:} $\deg(c_u) \ge 4$ and  $\deg(c_v) = \deg(c_w) = 2$. 
    
   We will first show that $Q$ is not a zero forcing set (and thus not a $1$-leaky forcing set). Applying Lemma \ref{lem:unitree} will force all $T_{c_u,j}$ and $c_u$. Thereafter, no forces are possible as $c_u$ has both remaining white vertices as neighbors. Hence $\Zone{G} > q$.
    
    We now show that $S=Q\cup \{c_w\}$ is a $1$-leaky forcing set. Suppose that the leak is in some $T_{c_u,j}$. Then applying Lemma \ref{lem:unitree} will force all  $T_{c_u,k}$  for $k \ne j$ and $c_u$. Thereafter, $c_w$ can force $c_v$, and applying Lemma \ref{lem:forcetree} will force $T_{c_u,j}$ and hence all of $G$.

    If the leak is on a cycle vertex, then applying Lemma \ref{lem:unitree} will force all $T_{c_u,j}$. Thereafter, either $c_u$ or $c_w$ will force $c_v$. Hence, $q < \Zone{G}\le q+1$, so $\Zone{G} = q+1$.

    \textbf{Case 4:} $\deg(c_u) \ge 3$, $\deg(c_v) = 3$, and  $\deg(c_w) = 2$.

    We will first show that $Q$ is not a $1$-leaky forcing set.
    Suppose the leak is within $T_{c_v, 1}$. Applying Lemma \ref{lem:unitree} will force all $T_{c_u,j}$ as well as $c_u$. However, $c_v$ will not be forced. Since forcing $c_w$ requires both $c_u$ and $c_v$ to be blue, $G$ will not be forced, so $\Zone{G} > q$.

    We now show that $S=Q\cup \{c_w\}$ is a $1$-leaky forcing set of $G$. If the leak is on $c_u, c_v$, or $c_w$, then Lemma \ref{lem:unitree} will guarantee that all of the $T_{c_u,j}$ and all of $T_{c_v,1}$ and both $c_u$ and $c_v$ will be forced.

    If the leak is within some $T_{c_u,j}$, say without loss of generality that it is in $T_{c_u,1}$. Then, by Lemma \ref{lem:unitree}, all of the $T_{c_v,j}$ (for $j\ne 1$), all of $T_{c_v,1}$, and $c_v$  will be forced. Hence, $c_w$ forces $c_u$ and by Lemma \ref{lem:forcetree}, $T_{c_u,1}$ will be forced (and hence all of $G$). The result follows similarly if the leak is within $T_{c_v,1}$. Hence $q < \Zone{G} \le q+1$, so $\Zone{G} = q+1$.

     \textbf{Case 5:} $\deg(c_u) = 3$ and  $\deg(c_w) = \deg(c_v) = 2$.

    We first argue that $\Zone{G}\ge q+2$. Suppose the leak is on the vertex in $T_{c_u,1}$ adjacent to $c_u$. Then upon coloring $Q$ blue, $T_{c_u,1}$ will be forced with $c_u, c_w, c_v$ remaining white. Further, if any one of $c_u, c_w$,  or $c_v$ are additionally blue, then no more forces can occur. It follows that $\{c_u, c_w\}, \{c_u, c_v\}$, and $\{c_v, c_w\}$ are each $1$-leaky forts. Hence, by Proposition \ref{prop:LeakyIFFforts}, at least two of $c_u, c_w$, and $c_v$ must in a 1-leaky forcing set, so $\Zone{G} \ge q+2$.

    Finally, we show $S=Q \cup \{c_w, c_v\}$ is a $1$-leaky forcing set of $G$. Necessarily, at least one $c_v$ or $c_w$ will force $c_u$. Thereafter, Lemma \ref{lem:forcetree} will guarantee that the remainder of $T_{c_u,1}$ will be forced. Hence, $\Zone{G} = q+2$.\end{proof}

Now we turn to the case of unicyclic graphs where the girth is at least 4.

\begin{thm} \label{thm:4+cycle}
 Let $G$ be a unicyclic graph with girth at least 4 and $Q = \{v\in V(G) : \deg(v)= 1 \}$ such that $|Q|=q$, then 
\[ \Zone{G} =\begin{cases}
     q & \text{if either} \\ & \text{\quad  there are at least two disjoint pairs of adjacent cycle vertices with degree at least 3 or} \\ & \text{\quad   there is a cycle vertex of degree at least 4 with either } \\ & \text{\quad \quad  at least one cycle neighbor has degree at least 4 or} \\ & \text{\quad\quad both cycle neighbors have  degree 3;} \\
     q+2 & \text{if either} \\ & \text{\quad all of the following apply:} \\ & \text{\quad \quad there are exactly two cycle vertices of degree $3$ and} \\ & \text{\quad \quad  they are adjacent and}\\
     &\text{\quad \quad all other cycle vertices have degree $2$ or} \\ &\text{\quad both of the following apply:} \\ &\text{\quad \quad every cycle vertex has degree at most 3 and} \\&\text{\quad\quad every pair of cycle vertices of degree 3 has distance at least 3 from each other;}\\
      q+1 & \text{otherwise.}
     \end{cases}
\]
\end{thm}

\begin{proof} Note by Lemma \ref{lem:lowdegree}, every $1$-leaky forcing set $S$ of $G$ must contain all the leaves of $G$, i.e., $Q\subseteq S$. Also, observe that for a vertex $c_i$ on the cycle $C$, if $c_i$ has degree $3$, placing the leak on the vertex of $T_{c_i,1}$ which is adjacent to $c_i$ prevents $c_i$ from being forced except perhaps by a cycle neighbor. If $c_i$ has degree $4$, by Lemma \ref{lem:unitree}, the leak cannot prevent it from being forced by some $T_{c_i,j}$.

By Lemma \ref{lem:unitree}, if the leak appears on any of the cycle vertices $c_i$, then all vertices of the graph will be forced to be blue except the cycle vertices of degree 2. In this case, the remaining cycle vertices can still be forced as long as there are two adjacent cycle vertices which are blue after the initial forcing in the trees. This happens by forcing in both directions around the cycle. While one direction of forcing may be interrupted by the leak, the other direction will continue until the cycle is completely forced. If the leak appears on a vertex of $T_{c_i,j}$, then applying Lemma \ref{lem:unitree} will force all $T_{c_i',j'}$ for all pairs $(i',j')\not=(i,j)$. In this case, all vertices of the graph will be blue except for the cycle vertices of degree $2$, and possibly some vertices in that particular $T_{c_i,j}\cup\{c_i\}$. Again, the remaining cycle vertices can still be forced as long as there are two adjacent cycle vertices which are blue after the initial forcing in the trees. After the cycle is forced, Lemma \ref{lem:forcetree} can be applied to force the remaining white vertices in  $T_{c_i,j}$.

Thus, to demonstrate that a set $S$ is a $1$-leaky forcing set, it remains to show that we can always find two adjacent blue cycle vertices after an initial round of forcing within the trees. We now consider several cases:

\textbf{Case 0:} Every cycle vertex has degree 2.

In this case, $G\cong C_n$ for which $\Zone{G} = 2$ by Proposition \ref{prop:basicgraphresults}. Then, since $q=0$, we have $\Zone{G} = q+2$.

For the remaining cases, we assume now that at least one cycle vertex has degree at least 3.

\textbf{Case 1:} $G$ contains at least two disjoint pairs of adjacent cycle vertices with degree at least 3.

We will show that $S=Q$ is a $1$-leaky forcing set. First, assume there are at least two disjoint pairs of adjacent vertices on $C$ with degree at least 3 and call them $c_u, c_{u+1}, c_v,$ and $c_{v+1}$. Note that $u+1<v$, since the pairs are disjoint. If the leak is on a cycle vertex, $c_u, c_{u+1}, c_v,$ and $c_{v+1}$ will be forced after applying Lemma \ref{lem:unitree}. Otherwise, suppose the leak appears on $T_{c_i,j}$. Then applying Lemma \ref{lem:unitree} will force all $T_{c_{i'},j'}$ for all pairs $(i',j')\not=(i,j)$, as well as $c_{i'}$ for $i'\neq i$. In particular, at least one pair $c_u$ and $ c_{u+1}$ or $c_v$ and $c_{v+1}$ will be forced. In both cases, we find two adjacent blue cycle vertices, as desired. Thus, $\Zone{G}=q$ in this case.

\textbf{Case 2:} $G$ does not contain two disjoint pairs of adjacent cycle vertices with degree at 3 or more and $G$ contains a vertex of degree $4$ or more.

\textbf{Case 2.1:} $G$ contains a pair of adjacent cycle vertices with degree 4 or more.

We will show that $S=Q$ is a $1$-leaky forcing set. Let $c_u$ and $c_{u+1}$ be a pair of adjacent cycle vertices such that $\deg(c_u)\geq 4$ and $\deg(c_{u+1})\geq 4$. Regardless of leak location, $c_u$ and $c_{u+1}$ will be forced by applying Lemma \ref{lem:unitree}. Since these vertices are adjacent blue cycle vertices, we see $\Zone{G}=q$, as desired.

\textbf{Case 2.2:} $G$ contains a cycle vertex of degree 4 or more which has two cycle neighbors with degree 3.

We will show that $S=Q$ is a $1$-leaky forcing set. Let $c_{u-1}$, $c_{u}$, and $c_{u+1}$ be a group of subsequent cycle vertices such that $\deg(c_u)\geq 4$ and $\deg(c_{u-1})=\deg(c_{u+1})=3$. Regardless of leak location, $c_u$ and at least one of $c_{u+1}$ or $c_{u-1}$ will be forced by applying Lemma \ref{lem:unitree}. Since we have found two adjacent blue cycle vertices, we see $\Zone{G}=q$, as desired.

\textbf{Case 2.3:} $G$ contains a cycle vertex of degree 4 or more has a cycle neighbor of degree $2$ and a cycle neighbor of degree $3$.

We will first show that $Q$ is not a $1$-leaky forcing set. Let $c_{u-1}$, $c_{u}$ and $c_{u+1}$ be a group of subsequent cycle vertices such that $\deg(c_u)\geq 4$, $\deg(c_{u-1})=3$, and $\deg(c_{u+1})=2$. Suppose the leak is on the vertex in $T_{c_{u-1},1}$ adjacent to $c_{u-1}$. Then, after applying Lemma \ref{lem:unitree}, $c_u$ will be forced but $c_{u-1}$ and $c_{u+1}$ will remain white. Thus, we cannot find a pair of adjacent blue cycle vertices inside the set $\{c_{u-1},c_{u},c_{u+1}\}$. Note that by assumption, $G$ cannot contain any pairs of adjacent cycle vertices with degree at least three outside of this set. Therefore every other pair of cycles vertices contains a vertex of degree 2, which will remain white. So we see $\Zone{G}> q$.

We now show that $S=Q \cup \{c_{u-1}\}$ is a $1$-leaky forcing set of $G$. Regardless of leak placement, $c_{u-1}$ and $c_u$ will be forced after applying Lemma \ref{lem:unitree}. Since we have found a pair of adjacent blue cycle vertices, we see $\Zone{G}=q+1$, as desired.

\textbf{Case 2.4:} Every cycle vertex of degree 4 or more has two cycle neighbors of degree $2$.

We will first show that $Q$ is not a $1$-leaky forcing set. The only way to obtain two adjacent blue cycle vertices after forcing in the trees is to have two adjacent cycle vertices with degree at least 3. In this case, such a pair of vertices must have degree exactly 3. By assumption, we do not have two disjoint pairs of degree 3 vertices. However, we could have one pair of adjacent vertices of degree 3 or one set of three subsequent vertices of degree 3. 

Let $c_{u-1}$, $c_{u}$ and $c_{u+1}$ be a group of subsequent cycle vertices such that $\deg(c_{u-1})\leq 3$ and $\deg(c_{u})=\deg(c_{u+1})=3$. Suppose the leak is on the vertex in $T_{c_{u},1}$ adjacent to $c_{u}$. 
 
 Assuming $\deg(c_{u-1})=3$, note that this implies $\deg(c_{u-2})=\deg(c_{u+2})=2$, since $G$ does not contain two disjoint pairs of adjacent cycle vertices with degree 3 or more. Then, after applying Lemma \ref{lem:unitree}, $c_{u-1}$ and $c_{u+1}$ will be forced but $c_{u-2}$, $c_{u}$, and $c_{u+2}$ will remain white. Thus, we cannot find a pair of adjacent blue cycle vertices inside the set $\{c_{u-2}, c_{u-1}, c_{u},c_{u+1},c_{u+2}\}$.

 If $\deg(c_{u-1})=2$, assume $\deg(c_{u+2})=2$ (or else we are in the case above). After applying Lemma \ref{lem:unitree}, $c_{u+1}$ will be forced but $c_{u-1}$, $c_{u}$, and $c_{u+2}$ will remain white. Thus, we cannot find a pair of adjacent blue cycle vertices inside the set $\{ c_{u-1}, c_{u},c_{u+1},c_{u+2}\}$.

 If $G$ does not contain any pairs of adjacent vertices of degree $3$, there will be no adjacent blue cycle vertices after applying Lemma \ref{lem:unitree}, regardless of the leak location. Therefore, in each of these cases, $\Zone{G}> q$.

Let $c_{v-1}$, $c_{v}$ and $c_{v+1}$ be a group of subsequent cycle vertices such that $\deg(c_v)\geq 4$ and $\deg(c_{v-1})=\deg(c_{v+1})=2$. We now show that $S=Q \cup \{c_{v-1}\}$ is a $1$-leaky forcing set of $G$. Regardless of leak placement, $c_{v-1}$ and $c_v$ will be blue after applying Lemma \ref{lem:unitree}. Since we have found a pair of adjacent blue cycle vertices, we see $\Zone{G}=q+1$, as desired.

\textbf{Case 3:} $G$ does not contain two disjoint pairs of adjacent cycle vertices with degree 3 or more and every vertex in $G$ has degree at most $3$.

\textbf{Case 3.1:} $G$ contains a set of three subsequent cycle vertices of degree $3$, $c_{u-1}, c_{u},$ and $c_{u+1}$.

We will first show that $Q$ is not a $1$-leaky forcing set. Note that $\deg(c_{u-2})=\deg(c_{u+2})=2$, since $G$ does not contain two disjoint pairs of adjacent cycle vertices with degree 3 or more.  Suppose the leak is on the vertex in $T_{c_{u},1}$ adjacent to $c_{u}$. Then, after applying Lemma \ref{lem:unitree}, $c_{u-1}$ and $c_{u+1}$ will be forced but $c_{u-2}$, $c_{u}$, and $c_{u+2}$ will remain white. Thus, we cannot find a pair of adjacent blue cycle vertices inside the set $\{c_{u-2}, c_{u-1}, c_{u},c_{u+1},c_{u+2}\}$. Note that by assumption, $G$ cannot contain any pairs of adjacent cycle vertices with degree at least three outside of this set. Therefore every other pair of cycles vertices contains a vertex of degree 2, which will remain white. So we see $\Zone{G}> q$.

We now show that $S=Q \cup \{c_{u}\}$ is a 1-leaky forcing set of $G$. Regardless of leak placement, at least one of $c_{u-1}$ and $c_{u+1}$ will be forced after applying Lemma \ref{lem:unitree}. Therefore, the set $\{c_{u-1},c_{u},c_{u+1}\}$ contains a pair of adjacent blue cycle vertices, and we see $\Zone{G}=q+1$, as desired.

\textbf{Case 3.2:} $G$ contains two vertices of degree $3$ with distance $2$ and their common neighbor has degree $2$.

We will first show that $Q$ is not a $1$-leaky forcing set. Let $c_{u-1}$ and $c_{u+1}$ be cycle vertices with distance $2$ such that $\deg(c_{u-1})=\deg(c_{u+1})=3$. Note that their common neighbor, $c_{u}$, must have degree $2$. Since $G$ does not contain two disjoint pairs of adjacent cycle vertices with degree 3 or more, at most one of $c_{u-2}$ and $c_{u+2}$ has degree $3$. If such a vertex of degree $3$ exists, suppose without loss of generality it is $c_{u-2}$. Suppose the leak is on the vertex in $T_{c_{u-1},1}$ adjacent to $c_{u-1}$. Then, after applying Lemma \ref{lem:unitree}, $c_{u+1}$ (and possibly $c_{u-2}$) will be forced but $c_{u-1}$, $c_{u}$, and $c_{u+2}$ will remain white. Thus, we cannot find a pair of adjacent blue cycle vertices inside the set $\{c_{u-2},c_{u-1}, c_{u},c_{u+1},c_{u+2}\}$. Note that outside of this set, $G$ cannot contain any pairs of adjacent cycle vertices with degree at least three. Therefore every other pair of cycles vertices contains a vertex of degree 2, which will remain white. So we see $\Zone{G}> q$.

We now show that $S=Q \cup \{c_{u}\}$ is a 1-leaky forcing set of $G$. Regardless of leak placement, at least one of $c_{u-1}$ and $c_{u+1}$ will be forced after applying Lemma \ref{lem:unitree}. Therefore, the set $\{c_{u-1},c_{u},c_{u+1}\}$ contains a pair of adjacent blue cycle vertices, and we see $\Zone{G}=q+1$, as desired.

\textbf{Case 3.3:} $G$ contains at least three cycle vertices of degree $3$ such that one pair of degree $3$ cycle vertices are adjacent and each remaining pair have distance at least three.

We will first show that $Q$ is not a $1$-leaky forcing set. Let $c_u$ and $c_{u+1}$ be the pair of adjacent cycle vertices such that $\deg(c_u)=\deg(c_{u+1})=3$. Suppose the leak is on the vertex in $T_{c_{u},1}$ adjacent to $c_{u}$. Then, after applying Lemma \ref{lem:unitree}, $c_{u+1}$ will be forced but $c_{u}$ will remain white. Furthermore, there will not be a pair of adjacent blue vertices anywhere else in the cycle since every other pair of degree $3$ cycle vertices has distance at least 3. Therefore, $\Zone{G}> q$.

Let $c_v$ be a third cycle vertex such that $\deg(c_v)=3$. We now show that $S=Q \cup \{c_{v+1}\}$ is a $1$-leaky forcing set of $G$. Regardless of leak location, after applying Lemma \ref{lem:unitree}, at least one pair $\{c_u,c_{u+1}\}$ or $\{c_v,c_{v+1}\}$ will be blue. Therefore, we have found a pair of adjacent blue cycle vertices, and we see $\Zone{G}=q+1$, as desired.

\textbf{Case 3.4:} $G$ contains exactly two cycle vertices of degree $3$ and they are adjacent.

We will first show that $Q$ is not a $1$-leaky forcing set. Let $c_u$ and $c_{u+1}$ be the pair of adjacent cycle vertices such that $\deg(c_u)=\deg(c_{u+1})=3$. Suppose the leak is on the vertex in $T_{c_{u},1}$ adjacent to $c_{u}$. Then, after applying Lemma \ref{lem:unitree}, $c_{u+1}$ will be forced but $c_{u}$ will remain white. Furthermore, there will not be a pair of adjacent blue vertices anywhere else in the cycle since every other cycle vertex has degree 2. Therefore, $\Zone{G}> q$.

To continue forcing, we must add at least one vertex to the forcing set. If a vertex in a tree $T_{c_i,1}$ for $i\in\{u,u+1\}$ is added to the forcing set, placing a leak on the vertex in $T_{c_{i},1}$ adjacent to $c_{i}$ still prevents $c_i$ from being forced. If a vertex on another tree $T_{c_i,1}$ for $i\not\in\{u,u+1\}$ is added, it may help force $c_i$, but $c_i$ will never be part of a pair of adjacent blue cycle vertices. Suppose instead that a cycle vertex were added to the forcing set and call this vertex $c_v$. If $c_v$ is not adjacent to $c_{u+1}$, placing the leak on the vertex in $T_{c_{u},1}$ adjacent to $c_{u}$ prevents $c_u$ from being forced. If $c_v$ is adjacent to $c_{u+1}$, placing the leak on the vertex in $T_{c_{u+1},1}$ adjacent to $c_{u+1}$ prevents $c_{u+1}$ from being forced. In all of these cases, we fail to find two adjacent blue cycle vertices, so $\Zone{G}> q+1$. 

To complete the proof, choose any two adjacent cycle vertices, $c_k$ and $c_{k+1}$, and add them to the zero forcing set. Regardless of their degree, these vertices are a pair of adjacent blue cycle vertices and they can force any remaining white cycle vertices. Therefore, $S=Q\cup\{c_k,c_{k+1}\}$ is a $1$-leaky forcing set and we see that $\Zone{G}= q+2$, as desired.

\textbf{Case 3.5:} Every pair of degree $3$ vertices in $G$ have distance at least three.

We will first show that $Q$ is not a $1$-leaky forcing set. Regardless of the location of the leak, no pair of adjacent blue cycle vertices exist after applying Lemma \ref{lem:unitree}. Therefore, $\Zone{G}> q$.

To continue forcing, we must add at least one vertex to the forcing set. If a vertex on a tree $T_{c_i,1}$ is added, it may help force $c_i$, but $c_i$ will never be part of a pair of adjacent blue cycle vertices. Suppose instead that a cycle vertex were added to the forcing set and call this vertex $c_v$. At most one neighbor of $c_v$ has degree $3$, since there are no degree $3$ vertices with distance two. Without loss of generality, if $c_v$ is adjacent to a degree 3 vertex $c_i$, suppose the leak is on the vertex in $T_{c_{i},1}$ adjacent to $c_{i}$. After applying Lemma \ref{lem:unitree}, $c_{i}$ will not be forced. Therefore, there is still no pair of adjacent blue cycle vertices after applying Lemma \ref{lem:unitree}. Thus, adding one additional vertex is not sufficient and $\Zone{G}> q+1$ in this case. 

To complete the proof, choose any two adjacent cycle vertices, $c_k$ and $c_{k+1}$. Consider the set $Q \cup \{c_k, c_{k+1}\}$. Regardless of their degree, the vertices $c_k$ and $c_{k+1}$ are a pair of adjacent blue cycle vertices and they can force any remaining white cycle vertices. Therefore, $S=Q\cup\{c_k,c_{k+1}\}$ is a $1$-leaky forcing set and we see that $\Zone{G}= q+2$, as desired.\end{proof}

Next, we prove our characterization for the $2$-leaky forcing number of unicyclic graphs. 
\begin{thm} \label{thm:uni2}
 Let $G$ be a unicyclic graph with girth $g$ and $R = \{v\in V(G) : \deg(v)\leq 2 \}$ such that $|R|=r$. Then 
\[\Ztwo{G} = \begin{cases}
     r+2 & \text{if $g=3,4$ and every cycle vertex has degree $3$,}\\
     r+1 & \text{if $g=3$ and two cycle vertices have degree $3$, or}\\
     &\text{if $g=4$ and one of the following applies:}\\
     &\text{\quad \quad $G$ contains two non-adjacent cycle vertices with degree $3$ or}\\ & \text{\quad \quad $G$ contains three cycle vertices of degree $3$,}\\
    r & \text{otherwise (including whenever $g \ge 5$).}
     \end{cases}
\]
\end{thm}

\begin{proof} 
By Lemma \ref{lem:lowdegree}, every vertex of degree at most two must be in every $2$-leaky forcing set $S$, i.e., $R\subseteq S$. Note that this includes all leaves, as well as vertices on the trees and cycle with degree 2. 

Observe that for a vertex $c_i$ on the cycle $C$, if $c_i$ has degree $3$, placing a leak on the vertex of $T_{c_i,1}$ which is adjacent to $c_i$ prevents $c_i$ from being forced except perhaps by a cycle neighbor. If $c_i$ has degree $4$, by Lemma \ref{lem:unitree}, one leak cannot prevent it from being forced by $T_{c_i,j}$ for some $j$. However, placing a leak on the vertices of $T_{c_i,1}$ and $T_{c_i,2}$ which are adjacent to $c_i$ prevents $c_i$ from being forced except perhaps by a cycle neighbor. If $c_i$ has degree $5$ or more, by Lemma \ref{lem:unitree}, placing two leaks on any $T_{c_i,j}$ cannot prevent $c_i$ from being forced.

Since $R\subseteq S$ and by Lemma \ref{lem:unitree}, if both leaks appear on any of the cycle vertices $c_i$, then all vertices of the graph will be forced, so we do not need to consider this case. 

If one of the leaks appears on a vertex of the tree $T_{c_i,j}$ and the other leak appears on the cycle vertex $c_{i'}$ (note this includes the case where $i=i'$), then applying Lemma \ref{lem:unitree} will force all $T_{c_k,m}$ for all pairs $(k,m)\not=(i,j)$. In this case, all vertices of the graph will be blue except for possibly some vertices in that particular $T_{c_i,j}\cup\{c_i\}$. If there are two adjacent blue cycle vertices at this point, then $c_i$ will be forced by one of its cycle neighbors (note that at most one such neighbor could be $c_{i'}$). At this point, the set of blue vertices is a $2$-leaky forcing set by Lemma \ref{lem:forcetree}. Therefore, the remaining vertices can be forced.

If the leaks appear on a vertex of the trees $T_{c_i,j}$ and $T_{c_{i'},j'}$ (note this includes the case where $i=i'$), then applying Lemma \ref{lem:unitree} will force all $T_{c_k,m}$ for all pairs $(k,m)\not\in\{(i,j),(i',j')\}$. In this case, all vertices of the graph will be blue except for possibly some vertices in those particular $T_{c_i,j}\cup T_{c_{i'},{j'}}\cup\{c_i\}\cup\{c_{i'}\}$. If there are two adjacent blue cycle vertices at this point, then $c_i$ and $c_{i'}$ will be forced by one of their cycle neighbors. At this point, the set of blue vertices is a $2$-leaky forcing set by Lemma \ref{lem:forcetree}. Therefore, the remaining vertices can be forced.

Thus, to demonstrate that a set $S$ is a $2$-leaky forcing set, it remains to show that when at least one leak is placed in the tree vertices, we can always find two adjacent blue cycle vertices after applying Lemma \ref{lem:unitree}. We now consider cases:

\textbf{Case 0:} Every cycle vertex has degree 2.

Observe $G\cong C_n$ and from Proposition \ref{prop:basicgraphresults}, $\Ztwo{G}=n=r$. 

For the remaining cases, we assume now that at least one cycle vertex has degree at least 3.

{\bf Case 1:} If $g=3$, let $c_u, c_v$ and $c_w$ be the cycle vertices.

\textbf{Case 1.1:} At most one cycle vertex has degree $3$.

We will show that $S=R$ is a $2$-leaky forcing set. Let $c_u$ be the cycle vertex which may have degree $3$. Note that $c_v$ and $c_w$ have degrees $2$, $4$, or at least $5$. If no cycle vertices have degree $4$, then regardless of the leak placement, $c_v$ and $c_w$ will either start in the forcing set or be forced when applying Lemma \ref{lem:unitree}. If at least one cycle vertex $c_i$, has degree $4$, placing the leaks on the vertices in $T_{c_i,1}$ and $T_{c_i,2}$ adjacent to $c_i$ prevents $c_i$ from being forced when applying Lemma \ref{lem:unitree}. However, the other two cycle vertices will either start in the forcing set or be forced. In any other leak placement $c_v$ and $c_w$ will either start in the forcing set or be forced after applying Lemma \ref{lem:unitree}. In each case, we find two adjacent blue cycle vertices, as desired. Thus, $\Ztwo{G}=r$.

\textbf{Case 1.2:} Exactly two cycle vertices, say $c_u$ and $c_v$, have degree $3$.

We will first show that $R$ is not a $2$-leaky forcing set for $G$. By placing leaks on the vertices in $T_{c_u,1}$ and $T_{c_v,1}$ adjacent to $c_u$ and $c_v$ respectively, both $c_u$ and $c_v$ will not be forced when applying Lemma \ref{lem:unitree}. Furthermore, $c_w$ has at least two white neighbors ($c_u$ and $c_v$), so it cannot force regardless of its degree. Therefore, $\Ztwo{G}> r$. 

We will now show that $S=R\cup\{c_u\}$ is a $2$-leaky forcing set. If $c_w$ has degree $5$ or more, after applying Lemma \ref{lem:unitree}, $c_w$ will be forced regardless of leak placement. If $c_w$ has degree $2$, it is already in the forcing set. In both of these cases, $c_w$ and $c_u$ are a pair of adjacent blue cycle vertices. If $c_w$ has degree $4$, then placing the leaks on $T_{c_w,1}$ and $T_{c_w,2}$ prevents $c_w$ from being forced when applying Lemma \ref{lem:unitree}. However, the other two cycle vertices will be forced, so $c_u$ and $c_v$ are a pair of adjacent blue cycle vertices. If the leaks are placed anywhere else, $c_w$ will be forced, so $c_w$ and $c_u$ are a pair of adjacent blue cycle vertices. Therefore, $\Ztwo{G}=r+1$.

\textbf{Case 1.3:} Every cycle vertex has degree $3$.

We will first show that $R$ is not a $2$-leaky forcing set for $G$. Placing the leaks on the vertices of $T_{c_u,1}$ and $T_{c_v,1}$ adjacent to $c_u$ and $c_v$ and applying Lemma \ref{lem:unitree} results in only one cycle vertex, $c_w$, being forced. Furthermore, $c_w$ has at least two white neighbors ($c_u$ and $c_v$), so it cannot force regardless of its degree. Therefore, $\Ztwo{G}> r$.

If a vertex of $T_{c_i,1}$ was added to the forcing set, then placing a leak on the vertex of $T_{c_i,1}$ adjacent to $c_i$ prevents $c_i$ from being forced when applying Lemma \ref{lem:unitree}. The other leak can be used to prevent one of the other cycle vertices from being forced. Suppose instead one of the cycle vertices, say $c_u$, was added to the forcing set. Then placing the leaks on the vertices of $T_{c_v,1}$ and $T_{c_w,1}$ adjacent to $c_v$ and $c_w$ prevents both of these vertices from being forced when applying Lemma \ref{lem:unitree}. In each of these cases, only one of the three cycle vertices will be forced. Therefore, no one additional vertex being added to the forcing set will be sufficient, so $\Ztwo{G}> r+1$.

The set $S=R\cup\{c_u,c_v\}$ is a $2$-leaky forcing set, since regardless of leak placement, we have found a pair of adjacent blue cycle vertices. Therefore, $\Ztwo{G}=r+2$.

\textbf{Case 2:} If $g=4$, let $c_u,c_v, c_w,$ and $c_z$ be the cycle vertices.

\textbf{Case 2.1:} $G$ contains at most one cycle vertex of degree $3$ or exactly two cycle vertices of degree $3$ and those vertices are adjacent.

We will show that $S=R$ is a $2$-leaky forcing set. In either case, there are two adjacent cycle vertices, say $c_u$ and $c_v$, which do not have degree $3$. If neither $c_u$ nor $c_v$ have degree $4$, then regardless of the leak placement, $c_u$ and $c_v$ will either start in the forcing set or be forced by one of their trees $T_{c_u,j}$ and $T_{c_v,k}$ when applying Lemma \ref{lem:unitree}. Therefore, we find two adjacent blue cycle vertices, as desired. If at least one of these vertices, say $c_u$, has degree $4$, placing the leaks on the vertices of $T_{c_u,1}$ and $T_{c_u,2}$ adjacent to $c_u$ prevents $c_u$ from being forced when applying Lemma \ref{lem:unitree}. However, in this case, the other three cycle vertices will be forced, so we find two adjacent blue cycle vertices in the set $\{c_v,c_w,c_z\}$, as desired. With any other leak configuration, $c_u$ and $c_v$ will be forced when applying Lemma \ref{lem:unitree}. In either case, we find two adjacent blue cycle vertices. Thus, $\Ztwo{G}=r$.

\textbf{Case 2.2:} $G$ contains exactly two cycle vertices of degree $3$ and these vertices are non-adjacent or $G$ contains exactly three cycle vertices of degree $3$. 

We will first show that $R$ is not a $2$-leaky forcing set. In either case, $G$ contains a pair of non-adjacent cycle vertices, say $c_u$ and $c_w$, with degree $3$. Placing the leaks on the vertices of $T_{c_u,1}$ and $T_{c_w,1}$ adjacent to $c_u$ and $c_w$ and applying Lemma \ref{lem:unitree} prevents $c_u$ and $c_w$ from being forced. Since $c_v$ and $c_z$ are each adjacent to at least two white vertices ($c_u$ and $c_w$), they cannot force, regardless of their degree. Therefore, $\Ztwo{G}> r$.

We will now show that $S=R\cup\{c_u\}$ is a $2$-leaky forcing set. Note that $c_u$ has a neighbor, say $c_v$, which does not have degree $3$. If $c_v$ does not have degree $4$, then it will either be in the forcing set or will be forced after applying Lemma \ref{lem:unitree}. In these cases, $c_u$ and $c_v$ are a pair of adjacent blue cycle vertices. If $c_v$ has degree $4$, then it can be prevented from being forced when applying Lemma \ref{lem:unitree} by placing the leaks on the vertices of $T_{c_v,1}$ and $T_{c_v,2}$ adjacent to $c_v$. However, in this case, all other cycle vertices will be forced. For any other leak placement, $c_v$ will be forced after applying Lemma \ref{lem:unitree}. In each case, we find two adjacent blue cycle vertices, so $\Ztwo{G}= r+1$.

\textbf{Case 2.3:} Every cycle vertex has degree $3$.

We will first show that $R$ is not a $2$-leaky forcing set. Let $c_u$ and $c_w$ be non-adjacent cycle vertices. Placing the leaks on the vertices of $T_{c_u,1}$ and $T_{c_w,1}$ adjacent to $c_u$ and $c_w$ and applying Lemma \ref{lem:unitree} prevents $c_u$ and $c_w$ from being forced. Since $c_v$ and $c_z$ are each adjacent to at least two white vertices ($c_u$ and $c_w$), they cannot force, regardless of their degree. Therefore, $\Ztwo{G}> r$.

If a vertex of $T_{c_i,1}$ was added to the forcing set, then placing a leak on the vertex of $T_{c_i,1}$ adjacent to $c_i$ prevents $c_i$ from being forced when applying Lemma \ref{lem:unitree}. The other leak can be used to prevent the vertex $c_j$ which is not adjacent to $c_i$ from being forced. Suppose instead one of the cycle vertices, say $c_v$, was added to the forcing set. Then placing the leaks on the vertices of $T_{c_u,1}$ and $T_{c_w,1}$ adjacent to $c_u$ and $c_w$ (where $c_u$ and $c_w$ are adjacent to $c_v$) prevents both of these vertices from being forced when applying Lemma \ref{lem:unitree}. In each of these cases, there will be two non-adjacent cycle vertices which are each adjacent to two white cycle vertices. Therefore, no one additional vertex being added to the forcing set will be sufficient, so $\Ztwo{G}> r+1$.
 
 Let $c_u$ and $c_v$ be adjacent cycle vertices. Then $S=R\cup\{c_u,c_v\}$ is a $2$-leaky forcing set, since regardless of leak placement, we have found a pair of adjacent blue cycle vertices. Therefore, $\Ztwo{G}=r+2$.

\textbf{Case 3:}  If $g\geq 5$.

We will show that $S=R$ is a $2$-leaky forcing set. The leaks can prevent at most two cycle vertices from being forced by their trees when applying  Lemma \ref{lem:unitree}. This means there must be at least two adjacent cycle vertices that are either in the zero forcing set or are forced by their trees. Therefore, $\Ztwo{G}= r$. \end{proof}

Finally, we finish our complete characterization of the leaky forcing number for unicyclic graphs by considering the $\ell\ge 3$ case.

\begin{thm} \label{thm:uni3+}
  Let $G$ be a unicylic graph. For $\ell\geq 3$, $$\Znum{G}{\ell}=|\{v\in V(G) : \deg(v)\leq \ell \}|.$$
\end{thm}

\begin{proof}
By Lemma \ref{lem:lowdegree}, every vertex of degree at most $\ell$ must be in every $\ell$-leaky forcing set $S$. Note that this includes all leaves, as well as vertices on the trees and cycle with degree at most $\ell$.

Observe that for a vertex $c_i$ on the cycle $C$, if $c_i$ has degree $\ell+1$, placing the leaks on the vertices of $T_{c_i,j}$ which are adjacent to $c_i$ prevents $c_i$ from being forced, except perhaps by a cycle neighbor. If $c_i$ has degree $\ell+2$, by Lemma \ref{lem:unitree}, $\ell-1$ leaks cannot prevent it from being forced by $T_{c_i,j}$ for some $j$. However, placing $\ell$ leaks on the vertices of $T_{c_i,j}$ which are adjacent to $c_i$ prevents $c_i$ from being forced, except perhaps by a cycle neighbor. If $c_i$ has degree $\ell+3$ or more, by Lemma \ref{lem:unitree}, placing leaks on any $T_{c_i,j}$ cannot prevent $c_i$ from being forced. 

We will now show $S= \{v\in V(G) : \deg(v)\leq \ell\} $ is an $\ell$-leaky forcing set. If $G$ contains a cycle vertex $c_i$ with degree $\ell+2$, observe that placing all $\ell$ leaks on the vertices of $T_{c_i,j}$ which are adjacent to $c_i$ prevents $c_i$ from being forced by its trees. However, applying Lemma \ref{lem:unitree}, every other cycle vertex $c_j$ will be forced by any of its trees (or $c_j\in S$). Because no leaks remain, $c_i$ can then be forced by one of its cycle neighbors. Then, applying Lemma \ref{lem:forcetree} results in the remaining vertices being forced. 

If $G$ contains a cycle vertex $c_i$ with degree $\ell+1$, observe that placing $\ell-1$ leaks on the vertices of $T_{c_i,j}$ which are adjacent to $c_i$ prevents $c_i$ from being forced by its trees. However, applying Lemma \ref{lem:unitree}, every other cycle vertex $c_j$ will be forced by its tree (or $c_j\in S$). Because only one leak remains, $c_i$ can then be forced by one of its cycle neighbors. Then, applying Lemma \ref{lem:forcetree} results in the remaining vertices being forced.

In both of these cases, any other leak placement will not prevent the cycle vertices (which are not already in $S$) from being forced when applying Lemma \ref{lem:unitree}. Similarly, if $G$ does not contain any cycle vertices with degree $\ell+1$ or $\ell+2$, all cycle vertices will either be in $S$ or will be forced by their trees when applying Lemma \ref{lem:unitree}, regardless of leak placement. Then, applying Lemma \ref{lem:forcetree} results in the remaining vertices being forced.

Therefore, $S=\{v\in V(G) : \deg(v)\leq \ell \}$ is an $\ell$-leaky forcing set, as desired.
\end{proof}

\section{Generalized Petersen Graphs} \label{sec:gp}

In this section, we determine the exact $\ell$-leaky forcing number for select generalized Peterson graphs when $\ell=1,2$ and all generalized Peterson graphs for $\ell\geq 3$. We also provide upper bounds for the $\ell$-leaky forcing number of generalized Petersen graphs for $\ell=1,2$.

A {\it generalized Petersen graph}, $P(n,k)$ has $V=\{x_1, x_2, \ldots, x_n, y_1, y_2, \ldots, y_n \}$ and 
\[E= \{y_iy_{i+1} : 1\leq i \leq n\} \cup \{x_iy_i : 1 \leq i \leq n\} \cup \{x_i x_{i+k} : 1 \leq i \leq n \}\]
where $n\ge 3$ and $1\leq k\leq \frac{n-1}{2}$. 
The upper bound for $k$ is consistent with previous studies on $P(n,k)$ (for example, see \cite{GPG20}) and has two considerations. First, $k = \frac{n}{2}$ when $n$ is even is disallowed, as then, the inner vertices would have degree 2 (instead of degree 3). Also, if $k > \frac{n}{2}$, we would have $P(n,n-k) \cong P(n,k)$; hence, the range of $k$ is restricted for simplicity. For this definition and throughout this section, the indices are taken modulo $n$; that is, $y_ny_1$, is an edge. The $x_i$ will be referred to as \textit{inner vertices} and the $y_i$ as \textit{outer vertices}. The classical Petersen graph is $P(5,2)$. Figure \ref{fig:genpet7,2} highlights examples of generalized Peterson graphs. 

\begin{figure}[h!]
\begin{center}
\begin{tabular}{ccc}

\begin{tikzpicture}[every node/.style={draw,circle,very thick}]
  \graph[clockwise, radius=2cm] {subgraph C_n [n=5,name=A] }; 
  \graph[clockwise, radius=1cm] {subgraph I_n [n=5,name=B] }; 
  \foreach \i in {1,2,...,5}{
    \draw (A \i) -- (B \i);
  }
  \foreach \i [evaluate={\j=int(mod(\i+2 -1,5)+1)}] 
     in {1,2,...,5}{
      \draw (B \j) -- (B \i);
    }
\end{tikzpicture}

&

\begin{tikzpicture}[every node/.style={draw,circle,very thick}]
  \graph[clockwise, radius=2cm] {subgraph C_n [n=7,name=A] }; 
  \graph[clockwise, radius=1cm] {subgraph I_n [n=7,name=B] }; 
  \foreach \i in {1,2,...,7}{
    \draw (A \i) -- (B \i);
  }
  \foreach \i [evaluate={\j=int(mod(\i+3 -1,7)+1)}] 
     in {1,2,...,7}{
      \draw (B \j) -- (B \i);
    }
\end{tikzpicture}

&

\begin{tikzpicture}[every node/.style={draw,circle,very thick}]
  \graph[clockwise, radius=2cm] {subgraph C_n [n=9,name=A] }; 
  \graph[clockwise, radius=1cm] {subgraph I_n [n=9,name=B] }; 
  \foreach \i in {1,2,...,9}{
    \draw (A \i) -- (B \i);
  }
  \foreach \i [evaluate={\j=int(mod(\i+4 -1,9)+1)}] 
     in {1,2,...,9}{
      \draw (B \j) -- (B \i);
    }
\end{tikzpicture}

\end{tabular}
\end{center}
\caption{Examples of generalized Petersen graphs: $P(5,2)$ (left), $P(7,3)$ (center), and $P(9,4)$ (right).}
\label{fig:genpet7,2}
\end{figure}
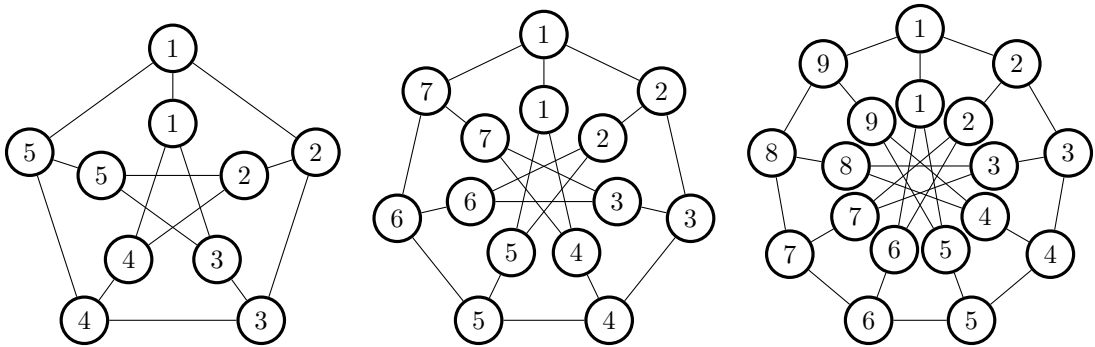

In this section, we will substantially build upon previous results for zero forcing and leaky forcing for generalized Petersen graphs. In \cite{rashidi2020computing}, the authors studied the zero forcing number of generalized Petersen graphs and we will make use of the following results from their paper.

\begin{thm}[\cite{rashidi2020computing}, Theorem 2.2] \label{thm:zg_gpg_2k+2} For any $n \ge 3$ and any $k \ge 1$, $\Zf{P(n,k)} \le 2k+2$.
\end{thm}

\begin{thm}[\cite{rashidi2020computing}, Theorems 2.5 and 3.6]\label{thm:ZF k=2}\label{thm:ZF k=3}
For $n\ge 10$, $\Zf{P(n,2)} = 6$ and for $n\ge 12$, $\Zf{P(n,3)} = 8.$
\end{thm}

In \cite{herrman2022d}, Herrman proved the following partial characterization of the $\ell$-leaky forcing for generalized Petersen graphs for the $k=1$ case.

\begin{thm} [\cite{herrman2022d}, Theorem 1.2]
\[\Zl{P(n,1)}= \begin{cases}
3 & \text{ if } \ell=0,1 \text{ and } n=3,\\
4 & \text{ if } \ell=0,1 \text{ and } n\geq 4 \text{ or } \ell=2 \text{ and } n=3,\\
2n & \text{ if } \ell\geq 3.
\end{cases}\]
\end{thm}

Note that since the minimum degree of $P(n,k)$ is always 3, applying Lemma \ref{lem:moreleaksinq} allows the following immediate generalization of Herrman's result for $\ell\geq 3$ to all $k$.

\begin{proposition}
For all $\ell\ge 3$ and $k\geq 1$, $\Zl{P(n,k)} = 2n$.
\end{proposition}

Since $\Zf{P(n,2)}$ and $\Zf{(P(n,3)}$ are known for $n\ge 10$ and $n\ge 12$ respectively (Theorem \ref{thm:ZF k=3}), we begin by computing several small cases when $\ell = 0,1,2$. These computational results can be verified by hand and were verified by computer using the algorithm which can be found in \cite{moruzziLeakyForcing} that utilizes Proposition \ref{prop:LeakyIFFforts}.

\begin{proposition} \label{prop:basicgpresultsk2}
For $\ell=0,1,2$, the values of $\Zl{P(n,2)}$ for $5\leq n\leq 9$ and the values of $\Zl{P(n,3)}$ for $7\leq n\leq 11$ are those stated in Table \ref{Tab:LowK2K3}.
\end{proposition}

\begin{table}[h!]
\begin{center}
\begin{tabular}{|c|c|c|c|}
 \hline
     $\Zl{P(n,2)}$ & $\ell=0$ & $\ell=1$ & $\ell=2$ \\
     \hline
    $n=5$ & 5&5&5\\
         \hline
    $n=6$ &4&6&6\\
         \hline
    $n=7$ &5&6&6\\
         \hline
    $n=8$ &5&5&7 \\
         \hline
    $n=9$ & 6&6&8\\
             \hline
\end{tabular}
\hspace{5mm}
\begin{tabular}{|c|c|c|c|}
 \hline
     $\Zl{P(n,3)}$ & $\ell=0$ & $\ell=1$ & $\ell=2$ \\
     \hline
    $n=7$ & 6&6&6\\
         \hline
    $n=8$ &6&6&8\\
         \hline
    $n=9$ &6&6&8\\
         \hline
    $n=10$ &8&8&8 \\
         \hline
    $n=11$ & 7&7&9\\
             \hline
\end{tabular}
\end{center}
\caption{Values of $\Zl{P(n,2)}$ for $5\leq n\leq 9$ and values of $\Zl{P(n,3)}$ for $7\leq n\leq 11$, for $\ell=0,1,2$.}
\label{Tab:LowK2K3}
\end{table}

For the remainder of this section, we focus on 1-leaky and 2-leaky forcing for ${P}(n,k)$.

\subsection{1-Leaky Forcing Number for Generalized Petersen Graphs}

We provide a generalization of Theorem \ref{thm:zg_gpg_2k+2} for $1$-leaky forcing. While we achieve the same upper bound, \cite{rashidi2020computing} uses the set $\{y_1, \ldots, y_{2k+2}\}$ which is a zero forcing set but not a $1$-leaky forcing set (try a leak at $y_k$). We will now define an alternative zero forcing set (see Figure \ref{fig:GP1Leak} for an example) and then show that it is also a $1$-leaky forcing set.

\newcommand\PeteN{15}
\newcommand\PeteK{3}
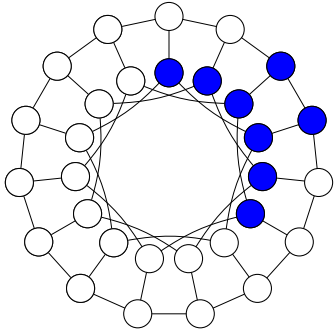
\begin{figure}[h!]
\begin{center}
\scalebox{.5}{

\begin{tikzpicture}[every node/.style={draw,circle,thick}]

  \graph[clockwise, radius=4cm] {subgraph C_n [n=\PeteN,name=A] };

  \graph[clockwise, radius=2.5cm] {subgraph I_n [n=\PeteN,name=B] }; 

  \foreach \i in {1,2,...,\PeteN}{
    \draw (A \i) -- (B \i);
  }

  \foreach \i [evaluate={\j=int(mod(\i+\PeteK -1,\PeteN)+1)}] in {1,2,...,\PeteN}{
    \draw (B \j) to [bend left=10] (B \i);
  }

  \foreach \j in {1,2,...,\PeteN}{ 
    \node[fill=white] at (A \j) {\phantom{00}};
    \node[fill=white] at (B \j) {\phantom{00}};
  }

  \foreach \j in {3,4}{
    \node[fill=blue] at (A \j) {\phantom{00}};
  }

  \foreach \j in {1,2,...,6}{
    \node[fill=blue] at (B \j) {\phantom{00}};
  }
\end{tikzpicture}
}
\end{center}
\caption{The generalized Peterson graph $P(15,3)$ with the $1$-leaky forcing set from Lemma \ref{lem:gponezfs} in blue.}
\label{fig:GP1Leak}
\end{figure}

\begin{lemma} \label{lem:gponezfs}
The set $S=\{x_1, x_2, \dots, x_{2k}, y_k, y_{k+1}\}$ is a zero forcing set of $P(n,k)$.
\end{lemma}

\begin{proof}
All vertices $y_j$ for 
$k+2 \le j \le 2k$
will be forced blue by vertices $y_{j-1}$, beginning with $y_{k+1}$ in a clockwise direction.
Next, we continue by induction.
Assume that all vertices 
$$x_{1}, x_{2} \dots, x_m \hspace{1cm} \textrm{and} \hspace{1cm} y_{k}, y_{k+1} \dots, y_{m}$$ are blue for $m \ge 2k$. Note that $y_{m+1}$ is the only white neighbor of the blue vertex $y_m$, hence will be forced by it. Since $m\ge 2k$ and $k\ge 1$, then $k\le m+1-k\le m$ and therefore $x_{m+1-k}$ and $y_{m+1-k}$ are blue. Also, $1\le m+1-2k\le m$, so $x_{m+1-2k}$ is blue. Therefore, $x_{m+1-k}$ can force $x_{m+1}$. By induction, this forcing in the clockwise direction will force all vertices of $P(n,k)$.
Note that the forcing can be done in a counterclockwise direction producing the same result.
\end{proof}

\begin{thm} \label{thm:gpone}
For $n \ge 2k$ and $k \ge 1$, $\Zone{{P}(n,k)} \le 2k+2$.
\end{thm}

\begin{proof}
Let $S=\{x_1, x_2, \dots, x_{2k}, y_k, y_{k+1}\}$. By applying Lemma \ref{lem:twodifferentleaky}, it suffices to show that each vertex not in $S$ can be forced in two ways. Observe that the symmetry of the proof of Lemma \ref{lem:gponezfs} provides that all vertices not in $S$ can forced either clockwise (as in the proof above) or counterclockwise. 
\end{proof}

\begin{thm} \label{thm:gponePn2exact}
For $n\geq 5$,
\[\Zone{P(n,2)}= \begin{cases}
5 & \text{ if } n=5,8,\\
6 & \text{ if } n=6,7 \text{ or } n \ge 9.
\end{cases}\]
\end{thm}

\begin{proof}
The result follows from Proposition \ref{prop:basicgpresultsk2} for $n\leq 9$ and from Theorems \ref{thm:ZF k=2} and \ref{thm:gpone} for $n\geq 10$.
\end{proof}

We now apply this upper bound to determine the exact $1$-leaky forcing number for the $k=1$ and $k=2$ cases.

\begin{thm}\label{thm:gponePn3exact}
For $n\geq 7$,
\[\Zone{P(n,3)}= \begin{cases}
6 & \text{ if } n=7,8,9\\
7 & \text{ if } n = 11 \\
8 & \text{ if } n =10 \text { or }n\geq 12.
\end{cases}\]
\end{thm}

\begin{proof}
The result follows from Proposition \ref{prop:basicgpresultsk2} for $n\leq 9$ and from Theorems \ref{thm:ZF k=2} and \ref{thm:gpone} for $n\geq 12$.
\end{proof}

\subsection{Upper Bound for 2-Leaky Forcing Number for Generalized Petersen Graphs}

In this subsection, we will provide an upper bound for the 2-leaky forcing number of $P(n,k)$ for $7 \le k \le \frac{n-1}{2}$ and $n\ge 6k+6$. Note that the bound $k\geq 7$ is required for the proof of Lemma \ref{lem:forcedownacrossandback}. Furthermore, the bound on $n$ in terms of $k$ is required for the proof of Lemma \ref{lem:alwaysforceuptoS}. In order to construct a 2-leaky forcing set, we introduce the following notation. 

Let us partition the vertices of $P(n,k)$ into blocks. We define the first \textit{block} to be $\blockb_1 := A_1 \cup B_1 \cup C_1 \cup D_1$, consisting of the following sets of vertices (see Figures \ref{fig:GP2Leak} and \ref{fig:GPFlat}):
\begin{multicols}{2}
\begin{itemize}
    \item[] $A_1 =  \{y_1, \ldots, y_{2k+2} \}$
    \item[] $B_1 = \{x_1, \ldots, x_{2k+2}\}$

    \columnbreak
    
    \item[] $C_1 =  \{y_{2k+3}, \ldots, y_{4k+4} \}$
    \item[] $D_1 = \{x_{2k+3}, \ldots, x_{4k+4}\}$
\end{itemize}
\end{multicols}

We iterate this description to define the $i$th block for $i \le  \lfloor \frac{n}{4k+4} \rfloor$ as follows:
\begin{multicols}{2}
\begin{itemize}
    \item[] $A_i =  \{y_{(4k+4)(i-1)+1}, \ldots, y_{(4k+4)(i-1)+2k+2} \}$
    \item[] $B_i = \{x_{(4k+4)(i-1)+1}, \ldots, x_{(4k+4)(i-1)+2k+2} \}$
    \columnbreak
    \item[] $C_i =  \{y_{(4k+4)(i-1)+2k+3}, \ldots, y_{(4k+4)(i-1)+4k+4} \}$
    \item[] $D_i = \{x_{(4k+4)(i-1)+2k+3}, \ldots, x_{(4k+4)(i-1)+4k+4} \}$
\end{itemize}
\end{multicols} 
For $\displaystyle i \le \left\lfloor \frac{n}{4k+4} \right\rfloor $, we denote the $i$-th block as $\blockb_i := A_i \cup B_i \cup C_i \cup D_i$.

A potential incomplete final block occurs when $4k+4$ does not evenly divide $n$. This block has index $i = \lceil \frac{n}{4k+4} \rceil$ and consists of all the remaining vertices as given above, excluding all vertices whose index exceeds $n$. Note that it may be the case that $C_{\lceil \frac{n}{4k+4} \rceil}$ and $D_{\lceil \frac{n}{4k+4} \rceil}$ are empty, in which case $A_{\lceil \frac{n}{4k+4} \rceil}$ may contain fewer than $2k+2$ vertices but will be directly adjacent to $A_1$.

Throughout this subsection, similar to the indices of the vertices, we will take the indices of the blocks $\blockb_i$ (and the sets $A_i, B_i, C_i$, and $D_i$) to be modulo $\lceil \frac{n}{4k+4} \rceil$.

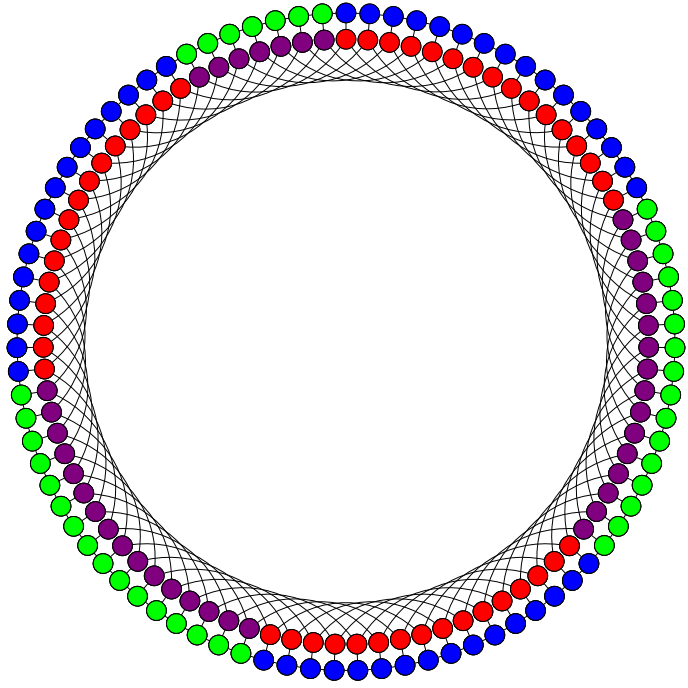
\begin{figure}[h!]
\centering 
\scalebox{.35}{
\begin{tikzpicture}[every node/.style={draw,circle,thin}]
  \renewcommand\PeteN{87}
  \renewcommand\PeteK{7}

  \graph[clockwise, radius=12.5cm] {subgraph C_n [n=\PeteN,name=A] };

  \graph[clockwise, radius=11.5cm] {subgraph I_n [n=\PeteN,name=B] }; 

  \foreach \i in {1,2,...,\PeteN}{
    \draw (A \i) -- (B \i);
  }

  \foreach \i [evaluate={\j=int(mod(\i+\PeteK -1,\PeteN)+1)}] in {1,2,...,\PeteN}{
    \draw (B \j) to [bend left=40] (B \i); 
  }

  \foreach \j in {1,2,...,\PeteN}{
    \node[fill=white] at (A \j) {\phantom{00}};
    \node[fill=white] at (B \j) {\phantom{00}};
  }

  \foreach \j in {1,2,...,16}{
    \node[fill=blue] at (A \j) {\phantom{00}};
  }
  \foreach \j in {1,2,...,16}{
    \node[fill=red] at (B \j) {\phantom{00}};
  }
  \foreach \j in {17,18,...,32}{
    \node[fill=green] at (A \j) {\phantom{00}};
  }
  \foreach \j in {17,18,...,32}{
    \node[fill=violet] at (B \j) {\phantom{00}};
  }
  \foreach \j in {33,34,...,48}{
    \node[fill=blue] at (A \j) {\phantom{00}};
  }
  \foreach \j in {33,34,...,48}{
    \node[fill=red] at (B \j) {\phantom{00}};
  }
  \foreach \j in {49,50,...,64}{
    \node[fill=green] at (A \j) {\phantom{00}};
  }
  \foreach \j in {49,50,...,64}{
    \node[fill=violet] at (B \j) {\phantom{00}};
  }
  \foreach \j in {65,66,...,80}{
    \node[fill=blue] at (A \j) {\phantom{00}};
  }
  \foreach \j in {65,66,...,80}{
    \node[fill=red] at (B \j) {\phantom{00}};
  }
  \foreach \j in {81,82,...,87}{
    \node[fill=green] at (A \j) {\phantom{00}};
  }
  \foreach \j in {81,82,...,87}{
    \node[fill=violet] at (B \j) {\phantom{00}};
  }
\end{tikzpicture}
}
\caption{The generalized Peterson graph $P(87,7)$, with $A_1,A_2, A_3$ in blue, $B_1,B_2, B_3$ in red, $C_1,C_2, C_3$ in green, and $D_1,D_2, D_3$ in purple. Note that $C_3$ and $D_3$ are incomplete due to the fact that $16$ does not evenly divide $87$. 
By Theorem \ref{thm:gptwo}, $A_1\cup A_2\cup A_3$ is a 2-leaky forcing set.}
\label{fig:GP2Leak}
\end{figure}

\begin{figure}[!ht]
\scalebox{0.62}{
\begin{tikzpicture}[node distance=0.5cm]
\node[circle, draw, fill=lightgray, minimum size=1cm, font=\footnotesize] (a1) {$y_1$};
\node[circle, draw, fill=lightgray, right=0.5cm of a1, minimum size=1cm, font=\footnotesize] (a2) {$y_2$};
\node[right=0.5cm of a2] (a1dots) {$\cdots$};
\node[circle, draw, fill=lightgray, right=0.5cm of a1dots, minimum size=1cm, font=\footnotesize] (ap1)  {$y_{k+1}$};
\node[circle, draw, fill=lightgray, right=0.5cm of ap1, minimum size=1cm, font=\footnotesize] (ap2)  {$y_{k+2}$};

\node[right=0.5cm of ap2] (a2dots) {$\cdots$};
\node[circle, draw, fill=lightgray, right=0.5cm of a2dots, minimum size=1cm, font=\footnotesize] (a2k1) {$y_{2k+1}$};
\node[circle, draw, fill=lightgray, right=0.5cm of a2k1, minimum size=1cm, font=\footnotesize] (a2k2) {$y_{2k+2}$};

\node[circle, draw, below=0.5cm of a1, minimum size=1cm, font=\footnotesize] (b1) {$x_1$};
\node[circle, draw, below=0.5cm of a2, minimum size=1cm, font=\footnotesize] (b2) {$x_2$};
\node[right=0.5cm of b2] (b1dots) {$\cdots$};
\node[circle, draw, below=0.5cm of ap1, minimum size=1cm, font=\footnotesize] (bp1) {$x_{k+1}$};
\node[circle, draw, below=0.5cm of ap2, minimum size=1cm, font=\footnotesize] (bp2) {$x_{k+2}$};
\node[right=0.5cm of bp2] (b2dots) {$\cdots$};
\node[circle, draw, below=0.5cm of a2k1, minimum size=1cm, font=\footnotesize] (b2k1) {$x_{2k+1}$};
\node[circle, draw, below=0.5cm of a2k2, minimum size=1cm, font=\footnotesize] (b2k2) {$x_{2k+2}$};

\draw (a1) -- (a2) -- (a1dots) -- (ap1) -- (ap2) -- (a2dots) -- (a2k1) -- (a2k2);

\draw (a1) -- (b1);
\draw (a2) -- (b2);
\draw (ap1) -- (bp1);
\draw (ap2) -- (bp2);
\draw (a1dots) -- (b1dots);
\draw (a2k1) -- (b2k1);
\draw (a2k2) -- (b2k2);
\node[circle, draw, right=0.5cm of a2k2, minimum size=1cm, font=\footnotesize] (c1) {$y_{2k+3}$};
\node[circle, draw, right=0.5cm of c1, minimum size=1cm, font=\footnotesize] (c2) {$y_{2k+4}$};
\node[right=0.5cm of c2] (c1dots) {$\cdots$};
\node[circle, draw, right=0.5cm of c1dots, minimum size=1cm, font=\footnotesize] (cp1) {$y_{3k+3}$};
\node[circle, draw, right=0.5cm of cp1, minimum size=1cm, font=\footnotesize] (cp2) {$y_{3k+4}$};
\node[right=0.5cm of cp2] (c2dots) {$\cdots$};
\node[circle, draw, right=0.5cm of c2dots, minimum size=1cm, font=\footnotesize] (c4k3) {$y_{4k+3}$};
\node[circle, draw, right=0.5cm of c4k3, minimum size=1cm, font=\footnotesize] (cn) {$y_{4k+4}$};

\node[circle, draw, right=0.5cm of b2k2, minimum size=1cm, font=\footnotesize] (d1) {$x_{2k+3}$};
\node[circle, draw, right=0.5cm of d1, minimum size=1cm, font=\footnotesize] (d2) {$x_{2k+4}$};
\node[right=0.5cm of d2] (d1dots) {$\cdots$};
\node[circle, draw, below=0.5cm of cp1, minimum size=1cm, font=\footnotesize] (dp1) {$x_{3k+3}$};
\node[circle, draw, below=0.5cm of cp2, minimum size=1cm, font=\footnotesize] (dp2) {$x_{3k+4}$};
\node[right=0.5cm of dp2] (d2dots) {$\cdots$};
\node[circle, draw, right=0.5cm of d2dots, minimum size=1cm, font=\footnotesize] (d4k3) {$x_{4k+3}$};
\node[circle, draw, right=0.5cm of d4k3, minimum size=1cm, font=\footnotesize] (dn) {$x_{4k+4}$};

\draw (a2k2) -- (c1) -- (c2) -- (c1dots) -- (cp1) -- (cp2) -- (c2dots) -- (c4k3) -- (cn);

\draw (c1) -- (d1);
\draw (c2) -- (d2);
\draw (cp1) -- (dp1);
\draw (cp2) -- (dp2);
\draw (c1dots) -- (d1dots);
\draw (c4k3) -- (d4k3);
\draw (cn) -- (dn);
\draw (c2dots) -- (d2dots);

\node[right=0.5cm of cn] (end1) {$\cdots$};
\node[right=0.5cm of dn] (end2) {$\cdots$};

\draw (cn) -- (end1);

\node[left=0.5cm of a1] (start1) {$\cdots$};
\node[left=0.5cm of b1] (start2) {$\cdots$};

\draw (a1) -- (start1);

\draw (d1) to[bend right=60] (dp1);
\draw (d2) to[bend right=60] (dp2);
\draw (dp1) to[bend right=60] (d4k3);
\draw (dp2) to[bend right=60] (dn);

\draw (b1) to[bend right=60] (bp1);
\draw (b2) to[bend right=60] (bp2);
\draw (bp1) to[bend right=60] (b2k1);
\draw (bp2) to[bend right=60] (b2k2);

\draw[decorate, decoration={brace, amplitude=20pt}] ([yshift=20pt]a1.north west) -- ([yshift=20pt]a2k2.north east) node[midway, above=30pt] {$A_i$};
\draw[decorate, decoration={brace, amplitude=20pt, mirror}] ([yshift=-40pt]b1.south west) -- ([yshift=-40pt]b2k2.south east) node[midway, below=30pt] {$B_i$};
\draw[decorate, decoration={brace, amplitude=20pt}] ([yshift=20pt]c1.north west) -- ([yshift=20pt]cn.north east) node[midway, above=30pt] {$C_i$};
\draw[decorate, decoration={brace, amplitude=20pt, mirror}] ([yshift=-40pt]d1.south west) -- ([yshift=-40pt]dn.south east) node[midway, below=30pt] {$D_i$};

\draw (a2dots) -- (b2dots);
\draw (b2k1) to[bend right=40] (d1dots.south west);
\draw (b2k2) to[bend right=50] (d1dots.south east);
\draw (d1) to[bend left=50] (b2dots.south west);
\draw (d2) to[bend left=40] (b2dots.south east);
\end{tikzpicture}
}

\caption{Schematic for the sets and vertex labels used within Section \ref{sec:gp}. The set labels are used in the proof of Theorem \ref{thm:gptwo}. The set of highlighted vertices are the 2-leaky forcing set used in Theorem \ref{thm:gptwo}.}
\label{fig:GPFlat}
\end{figure}
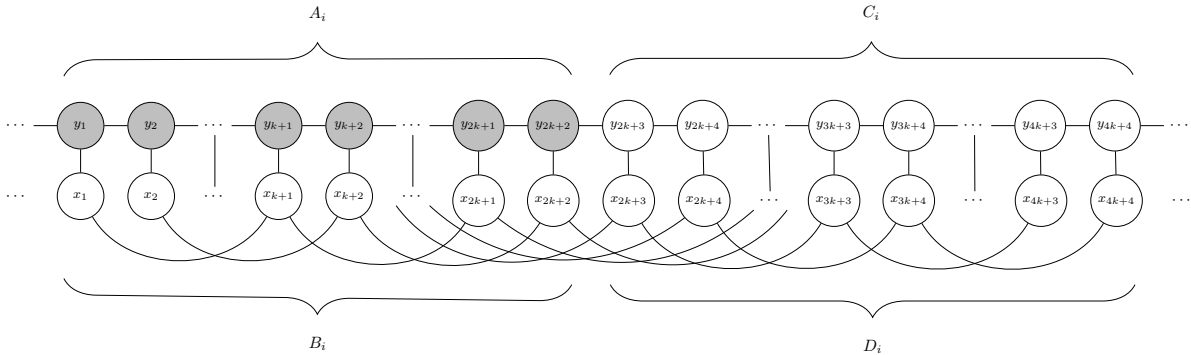

Before proving our upper bound for the 2-leaky forcing number, we require several lemmas. First, we show that each set $A_i$, which contains the full $2k+2$ vertices, is a zero forcing set.

\begin{lemma} \label{ZF} Each of the sets $A_i$, such that $|A_i|=2k+2$ (i.e.,  $1\leq i < \lceil \frac{n}{4k+4} \rceil$ and also perhaps $i = \frac{n}{4k+4}$ if $4k+4$ divides $n$) is a zero forcing set for $P(n,k)$. 
\end{lemma}

\begin{proof}
We will prove that $A_1$ is a zero forcing set. The blue vertices $y_i$, $2 \le i \le 2k+1$ will force vertices $x_i$, $2 \le i \le 2k+1$. Then, the forcing continues in both directions. Analogously, the remaining sets $A_i$, such that $2\leq i\leq \lceil \frac{n}{4k+4} \rceil$ and $|A_i|=2k+2$, are also zero forcing sets. 
\end{proof}

Note that Lemma \ref{ZF} implies that $S = \bigcup_{i} A_i$ is also a zero forcing set for $P(n,k)$. Next, we turn our attention to leaky zero forcing. In the proof of Theorem \ref{thm:gptwo}, we will demonstrate a 2-leaky forcing set by showing that for any one leak, every other vertex can be forced in at least two different ways (see Lemma \ref{lem:twodifferentleaky}). We will call the specific vertex that we are trying to force the {\it target} vertex. 
The next lemma shows that, provided the leak and target are within the specified block, all vertices with indices not between the indices of the two chosen vertices can be forced blue.

\begin{lemma}\label{lem:alwaysforceuptoS}
Let $S = \bigcup_{i} A_i$ in $P(n,k)$ and let $n\ge 6k+6$. Choose two distinct vertices, say $y_s$ or $x_s$ and either $y_{s'}$ or $x_{s'}$ with $1 \le s < s' \le 4k+4$. Then, all vertices $y_i, x_i$ with $1 \le i < s$ or $s' < i \le n$ (that are still white) can be forced by a vertex other than the chosen vertices.
\end{lemma}

\begin{proof}

Note that since $n\ge 6k+6$, $P(n,k)$ contains at least two $A_i$, $A_1$ and $A_2$, each containing $2k+2$ vertices. Since $A_2$ is a zero forcing set (by Lemma \ref{ZF}), it will force clockwise and counterclockwise and clockwise. Hence, all vertices $y_i, x_i$  with $1\le i < s$ and $s' < i \le n$ (that are still white) can be forced. In particular, all vertices $y_i, x_i$  with $1\le i < s$ can be forced from the left. Similarly, all vertices $y_i, x_i$  with $s' < i \le 4k+4$ (that are still white), can  be forced from the right. 
\end{proof}

As an immediate application of the previous lemma, we can see that $S=\bigcup_{i} A_i$ is also a $1$-leaky forcing set.

\begin{lemma}\label{lem:Sisa1-leakyFS}
Let $S = \bigcup_{i} A_i$ in $P(n,k)$ and let $n\ge 6k+6$. Then $S$ is a 1-leaky forcing set of $P(n,k)$.
\end{lemma}

\begin{proof}
The proof follows the proof of Lemma \ref{lem:alwaysforceuptoS}, taking $s=s'$.
\end{proof}

For each choice of leak and target vertex, the next lemma provides one method of forcing the target vertex. The vertices that are forcing the target are vertices whose indices are not in-between the indices of the leak and the target.

\begin{lemma}\label{lem:forcefromright}
Let $S  = \bigcup_{i} A_i$ in $P(n,k)$ and let $n\geq 6k+6$. Choose any vertex as a leak $y_f$ or $x_f$ with $1 \le f \le 4k+4$, and choose a target $y_t \notin A_1$ (resp. $x_t$) with $1 \le t \le 4k + 4$. If $f \le t$, $y_t$ (resp. $x_t$) can be forced by $y_{t+1}$ (resp. $x_{t+k}$). If $t < f$, $y_t$ (resp. $x_t$) can be forced by $y_{t-1}$ (resp. $x_{t-k}$).
\end{lemma}

\begin{proof}
Assume that $f \le t$.
By Lemma \ref{lem:alwaysforceuptoS} all vertices $y_i$, $x_i$, with $1 \le i <f$ and $t < i \le n$ can be forced by a vertex other than the chosen vertices. Then, since $f \le t$, $y_t$ (resp. $x_t$) will be forced by $y_{t+1}$ (resp. $x_{t+k}$). 

If $t < f$, the same proof applies where $y_t$ (resp. $x_t$) will be forced by $y_{t-1}$ (resp. $x_{t-k}$).
\end{proof}

The next lemma addresses how to force some of the $x_i$ vertices that are between the target and leak. Note that in the result, $s$ and $s'$ are interchangeable as the leak and the target. 

\begin{lemma}\label{lem:overcomebykminusone}
Let $S = \bigcup_{i} A_i$ in $P(n,k)$ and let $n\geq 6k+6$. Choose two distinct vertices, say $y_s$ or $x_s$ and either $y_{s'}$ or $x_{s'}$ with $1 \le s < s' \le 4k+5$. Suppose all vertices $y_i$ and $x_i$ with $1\le i < s$ or $i > s'$ are blue. Then, all vertices $x_{s+1}, x_{s+2}, \ldots, x_{s+k-1}$ and $x_{s'-1}, x_{s'-2}, \ldots, x_{s'-k+1}$ (that are not already blue) can still be forced to be blue without using the chosen vertices (e.g., either $y_s$ or $x_s$ nor either $y_{s'}$ or $x_{s'}$) to force.
\end{lemma}

\begin{proof}
By the hypothesis, each $x_i$ with $s-k+1 \le i \le s-1$ can force $x_{i+k}$. Similarly, $x_i$ with $s' + 1 \le  j \le s'+k-1$ can force $x_{j-k}$.   
\end{proof}

Our final lemma shows how the blue vertices from Lemma \ref{lem:overcomebykminusone} can force vertices in blocks $C_1$ and $D_1$. Note that while $t$ represents the target vertex, $s$ is a vertex who remains white because it would have been forced by the vertex with the leak.

\begin{lemma}\label{lem:forcedownacrossandback}
Let $S = \bigcup_{i} A_i$ in $P(n,k)$ be an initially blue set and let $7\leq k\leq \frac{n-1}{2}$ and $n\ge 6k+6$. Choose $t$ such that $2k+3 \le t \le 4k+5$. Assume there is one vertex $x_s$ with $1 \le s \le 4k+4$ that always remains white (or if $s=t$, is only forced by $x_{t-k}$) .
\begin{enumerate}
\item If $t - s\le 0$ then $y_{t}$ and $x_{t}$ can be forced by $y_{t-1}$ and $x_{t-k}$ respectively.
\item If $1 \le t-s <k$, then  $x_{t}$ and $y_{t}$ can be forced by $x_{t-k}$ and $x_{t}$ respectively.
\item If $t-s \ge k$, then, there exist two adjacent vertices $x_a, x_{a+1}$ (among $x_{t-k+2}$, \ldots, $x_{t-1}$) that can force $y_a$ and $y_{a+1}$. Thereafter, $y_{t}$ and $x_{t}$ can eventually be forced by $y_{t-1}$ and $y_{t}$ respectively.
\end{enumerate}
\end{lemma}

\begin{proof} Since $S$ is a zero forcing set, by Lemma \ref{lem:alwaysforceuptoS}, all vertices $y_i \notin S$ and $x_i$, with $1 \le i <\min\{s,t\}$ and $i >\max\{s,t\}$ can be forced without using any of the chosen vertices indexed by $s$ and $t$.

{\bf Case 1:} Follows immediately after the application of  Lemma \ref{lem:forcefromright}.

{\bf Case 2:} Note that $x_{t}$ can be forced by $x_{t-k}$ by Lemma \ref{lem:forcefromright} and then $x_{t}$ can force $y_{t}$.

{\bf Case 3:} By Lemma \ref{lem:overcomebykminusone}, 
 all vertices $x_{s+1}, x_{s+2}, \ldots, x_{s+k-1}$ and $x_{t-1}, x_{t-2}, \ldots, x_{t-k+1}$ (that are not already blue) can be forced without using the chosen vertices indexed by $s$ and $t$ to force.
 
 In any case, the vertices $x_i$, with $s < i \le 2k+1$, are forced by $y_i$. Hence, all vertices $x_{2k+2}, x_{2k+3}, \dots, x_{3k+1}$ (except the vertex $x_{s+k}$) can be forced by $x_{k+2}, x_{k+3}, \dots, x_{2k+1}$ respectively.

Also, the vertices $x_{t-k+1}, x_{t-k+2}, x_{t-k+3}$, and $x_{t-k+4}$ can be forced blue by $x_{t+1}, x_{t+2}, x_{t+3}$ and $x_{t+4}$ respectively.  Since $t \le 4k+5$,
the indices of the vertices $x_{t-k+1}, x_{t-k+2}, x_{t-k+3}$ and $x_{t-k+4}$ are at most $3k+6, 3k+7, 3k+8, 3k+9$ respectively. The vertices $x_{t-k+1}, x_{t-k+2}, x_{t-k+3}$, and $x_{t-k+4}$ are adjacent to $x_{t-2k+1 }, x_{t-2k+2}, x_{t-2k+3}$, and $x_{t-2k+4}$, the indices of which are at most $2k+6, 2k+7, 2k+8, 2k+9$. Since $k\ge 7$, these are less than or equal to $3k+2$. Since all but one of the vertices $x_{t-k+1}, x_{t-k+2}, x_{t-k+3}$, and $x_{t-k+4}$ have two blue neighbors, it follows that one pair among $\{ x_{t-k+1}, x_{t-k+2}\}$, $\{ x_{t-k+2}, x_{t-k+3}\}$, $\{ x_{t-k+3}, x_{t-k+4}\}$ has two blue neighbors. Call these vertices $x_a, x_{a+1}$.

The vertices $x_a, x_{a+1}$ can force $y_a, y_{a+1}$ respectively. Thereafter, since all $x_i$ with $i \ge a$ are blue, $y_{a+1}$ can force $y_{a+2}$, and inductively $y_{a+\ell}$ can force $y_{a+\ell+1}$ until $y_{t}$. Hence, $y_{t}$ will be forced by $y_{t-1}$. From there since $y_{t+1}$ is colored blue by Lemma \ref{lem:alwaysforceuptoS} and $y_{t-1}$ is blue, so $y_{t}$ forces $x_{t}$.
\end{proof}

We emphasize that the technique within the proof of Case 3 in Lemma \ref{lem:forcedownacrossandback} requires that $k \ge 7$ which is why our main result for this subsection, Theorem \ref{thm:gptwo} also requires $k \ge 7$ as a hypothesis. We will now state and prove our 2-leaky upper bound.

\begin{thm}  \label{thm:gptwo}
Let $7\leq k\leq \frac{n-1}{2}$ and $n\ge 10k+10$. Then $\Ztwo{P(n,k)}\leq \lceil \frac{n}{4k+4}\rceil (2k+2)$. 
\end{thm}

\begin{proof} Let $S = \bigcup_{i} A_i$ in $P(n,k)$ and {observe that $|S| \le \lceil \frac{n}{4k+4}\rceil (2k+2)$}.
By Lemma \ref{lem:Sisa1-leakyFS}, $S$ is a 1-leaky forcing set. Using Lemma \ref{lem:twodifferentleaky}, we will prove that $S$ is further a 2-leaky forcing set by showing that for any choice of a single leak, every other vertex not in $S$ can be forced in two different ways. We will refer to this vertex as the \textit{target}. For the remainder of the proof, we will index the leak as $x_f$ or $y_f$ and the target as $x_t$ or $y_t$.

Throughout the remainder of the proof, we will make use of the (near) symmetry of $P(n,k)$ by shifting and/or reversing the indices of the vertices and blocks. For instance, we may designate the set $A_2 \cup B_2 \cup C_1 \cup D_1$ as the first block by reindexing the vertices $6k+6 \to 1, 4k+5 \to 2k+2$, etc. When we do so, we will use the notation $\blockb'_i$ to emphasize the new block $i$. One may be concerned that doing so will move the final incomplete block to another location and disrupt the indexing of the vertices and the blocks. However, if there is a single leak within any block, then 
Lemma \ref{lem:alwaysforceuptoS} guarantees that all vertices up to (but not including) the leak and the target can be forced blue via the forcing sets $A_i$ from other blocks. Further, if $A_{\lceil \frac{n}{4k+4} \rceil}$ is truncated, it will be adjacent to $A_1$, thereby effectively making a longer string of initially blue vertices than necessary. Hence, any incomplete block will never increase the spacing of the white vertices in that block and the arguments at hand will not be affected.

For our first 12 cases, let us assume that the leak ($x_f$ or $y_f$) is in the same block as the target ($x_t$ or $y_t$). We will discuss the case where they are in different blocks later.

 Without loss of generality, we will assume that the leak and the target are in $\blockb_1$; 
that is,
$1 \le f \le 4k+4$ and $1 \le t \le 4k+4$. 
By Lemma \ref{lem:alwaysforceuptoS} all $y_i \notin S$ and $x_i$ for $1 \le i <f$ and $i >t$ can be forced. 
It remains to cover the 12 possible combinations of the leak (in $A_1, B_1, C_1$, or $D_1$) and the target ($B_1, C_1$, or $D_1$).

{\bf Case 1:} (Leak: $A_1$. Target: $B_1$.): \label{target:AB}
Let $y_f \in A_1$ and $x_{t} \in B_1$.
Assume by symmetry that $1 \le f \le t \le 2k+2$ (by relabeling and taking $\blockb_1' = A_1 \cup B_1 \cup C_{\lceil \frac{n}{4k+4}\rceil} \cup D_{\lceil \frac{n}{4k+4}\rceil}$, if necessary).
By Lemma \ref{lem:forcefromright}, $x_{t}$ can be forced by $x_{t+k}$. 
If $f \ne t$, then $x_{t}$ can also be forced by $y_{t}$. Alternatively, if $f=t$, then  $x_{t}$ can also be forced by $x_{t-k}$.

{\bf Case 2:}  (Leak: $A_1$. Target: $C_1$.): \label{target:AC}

Let $y_f \in A_1$ and  $y_{t} \in C_1$.
Then, by Lemma \ref{lem:forcefromright}, $y_{t}$ can be forced by $y_{t+1}$. By applying Lemma \ref{lem:forcedownacrossandback} with $s = f$, $y_{t}$ can also be forced by $y_{t-1}$ or $x_{t}$ (which one is determined by the difference $t-s$ as in Lemma \ref{lem:forcedownacrossandback}).

{\bf Case 3:} (Leak: $A_1$. Target: $D_1$.): \label{target:AD}

Let $y_f \in A_1$ and  $x_{t} \in D_1$.
Assume by symmetry that  $k+2 \le f \le 2k+2$ and $2k+3 \le t \le 4k+4$. 
Then by Lemma \ref{lem:forcefromright}, $x_{t}$ will be forced by $x_{t+k}$. On the other hand, by applying Lemma \ref{lem:forcedownacrossandback} with $s=f$, $x_{t}$ can also be forced by either $x_{t-k}$ or $y_{t}$ (which one is determined by the difference $t-s$ as in Lemma \ref{lem:forcedownacrossandback}).

{\bf Case 4:} (Leak: $B_1$. Target: $B_1$.): \label{target:BB}
Let $x_f \in B_1$ and  $x_{t} \in B_1$.
Assume by symmetry that $1 \le f \le t \le 2k+2$ (by relabeling as in Case 1). Then, by Lemma \ref{lem:forcefromright}, $x_t$ can be forced by $x_{t+k}$. 
It remains to show that $x_t$ can also be forced by $y_t$. 
If $t \ne 1, 2k+2$, this occurs automatically. If $f = t = 1$, then both $y_n$  (by Lemma \ref{lem:alwaysforceuptoS})  and $y_1$ (by hypothesis) will be blue, and then $y_1$ can force $x_1$. Similarly, for $t = 2k+2$, both 
$y_{2k+1}$ (by hypothesis) and $y_{2k+3}$ (by Lemma \ref{lem:alwaysforceuptoS}) will be blue; thus,
$y_{2k+2}$ can force $x_{2k+2}$.

{\bf Case 5:} (Leak: $B_1$. Target: $C_1$.): \label{target:BC}
Let $x_f \in B_1$ and $y_{t} \in C_1$. 
Note that by Lemma \ref{lem:forcefromright}, $y_{t}$ can be forced via $y_{t+1}$. On the other hand, by applying Lemma \ref{lem:forcedownacrossandback} with $s=f+k$, $y_{t}$ can also be forced by either $y_{t-1}$ or $x_{t}$ (which one is determined by the difference $t-s$ as in Lemma \ref{lem:forcedownacrossandback}).

{\bf Case 6:} (Leak: $B_1$. Target: $D_1$.):\label{target:BD}
Let $x_f \in B_1$ and $x_t\in D_1$. 
Note that by Lemma \ref{lem:forcefromright}, $x_{t}$ can be forced via $x_{t+k}$. On the other hand, by applying Lemma \ref{lem:forcedownacrossandback} with $s = f+k$, $x_{t}$ can also be forced by $y_{t}$ or $x_{t-k}$ (which one is determined by the difference $t-s$ as in Lemma \ref{lem:forcedownacrossandback}).

{\bf Case 7:} (Leak: $C_1$. Target: $B_1$.): \label{target:CB}
Let $y_f \in C_1$ and  $x_{t} \in B_1$. By Lemma \ref{lem:alwaysforceuptoS}, $x_{t}$ can be forced by $x_{t-k}$.
If $t \ne 1, 2k+2$, then $y_t$ will force $x_t$ naturally. If $t=1$, Lemma \ref{lem:alwaysforceuptoS} will force $y_n$, so $y_1 = y_t$ will force $x_1 = x_t$. 

If $t = 2k+2$, we proceed similarly to the proof of Lemma \ref{lem:forcedownacrossandback}, Case 3. By Lemma
\ref{lem:overcomebykminusone}, all vertices $x_{2k+3}, \ldots, x_{3k+1}$ and $x_{f-k+1}, \ldots, x_{f-1}$ can be forced. Then, $x_f$ can be forced by $x_{f+k}$.
Necessarily, $f \le 4k+4$, so $f-k+1$ is at most $3k+5$. Since $k \ge 7$, $x_{3k}$  and $x_{3k+1}$ can force $y_{3k}$ and $y_{3k+1}$
respectively (as $x_{2k}$, $x_{2k+1}$, $x_{4k}$, and $x_{4k+1}$ must be blue).
Thereafter,  $y_{3k}$ can force $y_{3k-1}$, and inductively, $y_{i}$ can force $y_{i-1}$ from $i=3k$ until $i=2k+3$. Finally, $y_{2k+2}=y_t$ can force $x_{2k+2}=x_t$.

{\bf Case 8:} (Leak: $C_1$. Target: $C_1$.):\label{target:CC}
Let $y_f \in C_1$ and $y_{t} \in C_1$. Assume by symmetry that $2k+3 \le f \le t \le 4k+4$ 
(by relabeling and taking $\blockb_1' = A_2 \cup B_2 \cup C_{1} \cup D_{1}$, if necessary).
Then by Lemma \ref{lem:forcefromright}, $y_{t}$ can be forced by $y_{t+1}$. On the other hand, by applying Lemma \ref{lem:forcedownacrossandback} with $s=f+k$, $y_{t}$ can also be forced by either $y_{t-1}$ or $x_{t}$ (which one is determined by the difference $t-s$ as in Lemma \ref{lem:forcedownacrossandback}).

{\bf Case 9:} (Leak: $C_1$. Target: $D_1$.): \label{target:CD}
Let $y_f \in C_1$ and $x_{t} \in D_1$. Assume by symmetry that $2k+3 \le f \le t \le 4k+4$ (by relabeling as in Case 8).
~Then by Lemma \ref{lem:forcefromright}, $x_{t}$ can be forced by $x_{t+k}$. On the other hand, by applying Lemma \ref{lem:forcedownacrossandback} with $s=f+k$, $x_{t}$ can also be forced by either $x_{t-k}$ or $y_{t}$ (which one is determined by the difference $t-s$ as in Lemma \ref{lem:forcedownacrossandback}).
  
{\bf Case 10:} (Leak: $D_1$. Target: $B_1$.): \label{target:DB}
Let $x_{f} \in D_1$ and $x_t \in B_1$. By Lemma \ref{lem:forcefromright}, $x_t$ can be forced by $x_{t-k}$. If $2 \le t \le 2k+1$ then $x_t$ can also be forced naturally by $y_t \in A_1$. If $t=1$, Lemma \ref{lem:alwaysforceuptoS} will force $y_n$, so $y_1 (= y_t)$ will force $x_1 (= x_t)$.

If $t = 2k+2$, we proceed similarly to the proof of Lemma \ref{lem:forcedownacrossandback}, Case 3. By Lemma \ref{lem:alwaysforceuptoS},
all vertices $x_i$ and $y_i$ such that $1\leq i< t$ and $f<i\leq n$ (that are still white) can be forced. By Lemma
\ref{lem:overcomebykminusone}, all vertices $x_{2k+3}, \ldots, x_{3k+1}$ and $x_{f-k+1}, \ldots, x_{f-1}$ can be forced. Then, $x_f$ can be forced by $x_{f+k}$.
Necessarily, $f \le 4k+4$, so $f-k+1$ is at most $3k+5$. Since $k \ge 7$, $x_{3k}$  and $x_{3k+1}$ can force $y_{3k}$ and $y_{3k+1}$
respectively (as $x_{2k}$, $x_{2k+1}$, $x_{4k}$ and $x_{4k+1}$ must be blue).
Thereafter,  $y_{3k}$ can force $y_{3k-1}$, and inductively, $y_{i}$ can force $y_{i-1}$ from $i=3k$ until $i=2k+3$. Finally, $y_{2k+2}=y_t$ can force $x_{2k+2}=x_t$.

{\bf Case 11:} (Leak: $D_1$. Target: $C_1$.): \label{target:DC}
Let $x_f \in D_f$ and $y_{t} \in C_1$. Assume by symmetry that $2k+3 \le f \le t \le 4k+4$ (by relabeling as in Case 8). ~Then by Lemma \ref{lem:forcefromright}, $y_{t}$ can be forced by $y_{t+1}$. On the other hand, by applying Lemma \ref{lem:forcedownacrossandback} with $s = f + k$, $y_{t}$ can also be forced by $y_{t-1}$ or $x_{t}$ (which one is determined by the difference $t-s$ as in Lemma \ref{lem:forcedownacrossandback}).

{\bf Case 12:} (Leak: $D_1$. Target: $D_1$.): \label{target:DD}
Let $x_f \in D_1$ and $x_{t} \in D_1$. Assume by symmetry that $2k+3 \le f \le t \le 4k+4$ (by relabeling as in Case 8).
Then by Lemma \ref{lem:forcefromright}, $x_{t}$, can be forced by $x_{t+k}$. On the other hand, by applying Lemma \ref{lem:forcedownacrossandback} with $s = f+k$, $x_{t}$ can also be forced by $x_{t-k}$ or $y_{t}$ (which one is determined by the difference $t-s$ as in Lemma \ref{lem:forcedownacrossandback}).

We now assume that the leak ($x_f$ or $y_f$) and the target ($x_t$ or $y_t$) are in different blocks.
The proof requires the existence of three full $A_i$ which is guaranteed as $n \ge 10k+10$.

 Let $\blockb_{f^*}$ and $\blockb_{t^*}$ denote the blocks containing the leak and target respectively.  If $f^* \ne t^* \pm1$ (i.e., the blocks are not adjacent), by Lemma $\ref{ZF}$, $A_{t^*-1}$,  $A_{t^*}$, and $A_{t^*+1}$ are each a zero forcing sets that can force all of $\blockb_{t^*}$ and from different directions. Hence, for the remainder of the proof, we will assume that $\blockb_{f^*}$ and $\blockb_{t^*}$ are adjacent.
 By symmetry, we will take $t^*=2$ and $f^*=1$ or $3$; this leaves the following three cases.

{\bf Case 13:} (Leak: $\blockb_1 \cup \blockb_3$, Target: $B_2$):

{\bf Case 13a:} (Leak: $\blockb_1 \cup \blockb_3$, Target: $B_2$ with $x_t \ne x_{4k+5}$ or $x_{6k+6}$):

The target $x_t$ can be forced directly by $y_t \in A_2$. Also, one of $\blockb_1 \cup \blockb_3$ does not have a leak. So by Lemma \ref{lem:forcefromright}, either $x_{t - k}$ or  $x_{t+k}$ can force $x_t$, depending on if the leak is in $\blockb_3$ or $\blockb_1$, respectively.

{\bf Case 13b:} (Leak: $\blockb_1$, Target: $B_2$ with $x_t = x_{4k+5}$):

By Lemma \ref{lem:forcefromright}, $x_{t}$ can be forced by $x_{t+k}$. On the other hand, by applying Lemma \ref{lem:forcedownacrossandback} with $s=f+k$ and $t = 4k+5$, $x_{t}$ can also be forced by either $x_{t-k}$ or $y_{t}$ (which one is determined by the difference $t-s$ as in Lemma \ref{lem:forcedownacrossandback}).

{\bf Case 13c:} (Leak: $\blockb_3$, Target: $B_2$ with $x_t = x_{4k+5}$):

By Lemma \ref{lem:alwaysforceuptoS}, $A_1$ can eventually force all of $\blockb_1$. Thereafter, $x_t$ can be forced by $y_{t}$. Alternatively, $x_t$ can be forced by $x_{t - k}$.

{\bf Case 13d:} (Leak: $\blockb_1$, Target: $B_2$ with $x_t = x_{6k+6}$):

$A_3$ can eventually force all of $\blockb_3$. Thereafter, by Lemma \ref{lem:alwaysforceuptoS}, all of $C_2 \cup D_2$ can be forced. Then, $x_t$ can be forced by $y_{t}$. Alternatively, $x_t$ can be forced by $x_{t + k}$.

{\bf Case 13e:} (Leak: $\blockb_3$, Target: $B_2$ with $x_t = x_{6k+6}$):

This case follows from symmetry upon relabeling the vertices so that $6k+6 \to 4k+5$ and $4k+5 \to 6k+6$. Hence, we define the new blocks $\blockb'_2 := A_2 \cup B_2 \cup C_1 \cup D_1$, $\blockb'_1 := A_3 \cup B_3 \cup C_2 \cup D_2$, and $\blockb'_0 := A_4 \cup B_4 \cup C_3 \cup D_3$.

If the leak is in $\blockb'_2$, apply Case 7 or 10 as appropriate. If the leak is in $\blockb'_1$, apply Case 13b. If the leak is in $\blockb'_0$, then the leak and the target are no longer considered to be in adjacent blocks, a case we have already covered (see above Case 13).

{\bf Case 14:} (Leak: $\blockb_1$, Target: $C_2 \cup D_2$ ):

By Lemma \ref{ZF}, $A_2$ can force all of $C_2 \cup D_2$ with either $y_{t-1}$ forcing $y_{t}$ or $x_{t-k}$ forcing $x_{t}$. Similarly, $A_3$ can force $C_2 \cup D_2$ with either $y_{t+1}$ forcing $y_{t}$ or $x_{t+k}$ forcing $x_{t}$.

{\bf Case 15:} (Leak: $\blockb_3$, Target: $C_2 \cup D_2$ ):

If the leak is in $A_3 \cup B_3$, then relabel the vertices so that the first block is now $\blockb'_1 =  A_3 \cup B_3 \cup C_2 \cup D_2$ and apply Cases 2, 3, 5, or 6 as appropriate. If the leak is in $C_3 \cup D_3$, then by Lemma \ref{ZF}, $A_2$ can force all of $C_2 \cup D_2$ with either $y_{t-1}$ forcing $y_{t}$ or $x_{t-k}$ forcing $x_{t}$. Similarly, $A_3$ can force $C_2 \cup D_2$ with either $y_{t+1}$ forcing $y_{t}$ or $x_{t+k}$ forcing $x_{t}$.
\end{proof}

\section{Vertex and Edge Removal} \label{sec:ddrowgen}

It was shown in \cite{edholm2012vertex} that for any graph $G$ and any vertex $v$ or edge $e$, the following results are valid for zero forcing: $-1 \le \Zf{G}-\Zf{G-e} \le 1$ and  $-1 \le \Zf{G}-\Zf{G-v} \le 1$. We provide analogous results for leaky forcing and show these bounds are tight.

\begin{lemma} \label{lem:gtog-efort}
   Let $G$ be a graph and let $\ell \ge 0$. If $F$ is an $\ell$-leaky fort of $G-e$ (respectively $G$), then either $F$ is an $\ell$-leaky fort of $G$ (respectively $G-e$) or $e$ intersects $F$. 
\end{lemma}

\begin{proof}
For either respective case, if neither endpoint of $e$ is in $F$, then, for any vertex $u \in G- F$,
the set of edges between $u$ and $F$ is the same in both $G$ and $G-e$.
Hence, $F$ is an $\ell$-leaky fort of $G$ (respectively $G-e$).
\end{proof}

\begin{thm} \label{thm:edgeremove} Let $G$ be a graph and let $\ell \ge 0$. For any edge $e$, $$-2\leq \Zl{G}-\Zl{G-e}\leq 2$$
and these bounds are tight for any $\ell\ge 1$.
\end{thm}

\begin{proof}
Let $e=v_1v_2$. If $S'$ is a minimum $\ell$-leaky forcing set of $G-e$ (i.e., $|S'| = \Zl{G-e}$), then by Proposition \ref{prop:LeakyIFFforts}, $S'$ intersects every $\ell$-leaky fort of $G - e$.  However, by Lemma \ref{lem:gtog-efort} we have that every $\ell$-leaky fort of $G$ either intersects $e$ or is an $\ell$-leaky fort of $G-e$. It follows that $S' \cup  \{v_1,v_2\}$ intersects every $\ell$-leaky fort of $G$. Hence, by minimality, $|S' \cup \{v_1,v_2\}| \ge  \Zl{G}$ and so $\Zl{G} \le |S' \cup \{v_1,v_2\}| \le \Zl{G-e} +2$. Therefore, $\Zl{G}-\Zl{G-e}\geq -2$.

Now, we mostly repeat the previous argument, exchanging $G$ and $G-e$ as follows. If $S$ is a minimum $\ell$-leaky forcing set of $G$, then $S$ intersects every $\ell$-leaky fort of $G$. By Lemma \ref{lem:gtog-efort}, we have that every $\ell$-leaky fort of $G-e$ either intersects $e$ or is an $\ell$-leaky fort of $G$. It follows that $S \cup \{v_1,v_2\}$ intersects every $\ell$-leaky fort of $G-e$. Hence, by minimality, $|S \cup \{v_1,v_2\}| \ge  \Zl{G - e}$ and so $\Zl{G-e)} \le |S \cup \{v_1,v_2\}| \le \Zl{G} +2$; hence, $\Zl{G-e}- \Zl{G}\le 2$.

We will now demonstrate the bound is tight in both directions.~\\
~\\
\noindent {\it Example for $\Zl{G}-\Zl{G-e} = -2$}.

Choose $\ell \ge 1$. Consider $2S_{\ell+1}$, two disjoint copies of the star graph, each with $\ell+1$ vertices (and hence, each with $\ell$ leaves). Let $G$ be the graph constructed by adding the edge $e$ between the two central vertices of each star. Since $G$ is a tree and $G-e$ is a forest, by Theorem \ref{thm:LeakyTree}, the forcing sets are precisely the set of vertices of degree $\ell$ or less, so we have $\Zl{G}-\Zl{G-e} = 2 \ell - (2\ell+2) = -2$
~\\
\noindent {\it Example for $\Zl{G}-\Zl{G-e} = 2$}.

Choose $\ell \ge 1$.

For $\ell = 1$, consider $G = P_6 + e$ the path graph on $6$ vertices  with an edge added between the second and fifth vertices. One can check that $\Zone{G} = 4$ whereas $\Zone{G-e} = \Zone{P_6} = 2$.

For $\ell \ge 2$, consider the graph $G$ illustrated in Figure \ref{fig:gre}, constructed as follows:  Take the complete graph on $4$ vertices that come in $2$ pairs: $v_{1,1}, v_{1,2}, v_{2,1}, v_{2,2}$. 
Add $\lfloor \ell/2 \rfloor$ leaves to each $v_{1,1}$ and $v_{2,1}$, and add 
$\lceil \ell/2 \rceil$ leaves to each $v_{2,1}$ and $v_{2,2}$. Finally. for each $i=1,2$, add $\lfloor \ell/2 \rfloor$ vertices $w_{i,1}, \ldots, w_{i,\lfloor \ell/2 \rfloor}$ vertices that are adjacent to both $v_{i,1}$ and $v_{i,2}$.

Let $B$ be the set of vertices with degree $\ell$ or less (i.e, all of the leaves and the $w_{i,j}$).

We will show that $\Zl{G} = |B| + 2$.

First, let us show that the set $B \cup \{v_{1,1}, v_{2,1}\}$ is a 
$\ell$-leaky forcing set. Notice that for any fixed $i'=1,2$, the only way to prevent $v_{i',2}$ from being forced is to place all $\ell$ leaks on the leaves of the $v_{i',2}$ {\it and} the $w_{i',j}$. Since this would requires all $\ell$ leaks, there are only sufficiently many leaks to prevent one $v_{i,2}$ from being forced. Thereafter, $v_{1,1}$ or $ v_{2,1}$ can force $v_{i',2}$.

Next, to see this is minimum 
first note that by Lemma \ref{lem:lowdegree}, $B$ must be a subset of any $\ell$-leaking forcing 
set. Second, observe that $\{v_{1,1}, v_{1,2}\}$ (respectively $\{v_{2,1}, v_{2,2}\}$) is a fort as there are exactly $\ell$ vertices with one neighbor in $\{v_{1,1}, v_{1,2}\}$ (respectively $\{v_{2,1}, v_{2,2}\}$) (i.e., the $\ell$ leaves for that pair). By Proposition \ref{prop:LeakyIFFforts}, every $\ell$-leaky forcing set must intersect each fort. Since these two forts are disjoint and are disjoint from $B$, the zero forcing set must contain at least two additional vertices, one from each fort. Therefore, $\Zl{G} \ge |B| + 2$, and since we have demonstrated such a forcing set, we have $\Zl{G} = |B| + 2$.

Now consider $G-e$ where $e$ is the edge $v_{1,1}v_{2,1}$. 

It suffices to show that $B$ is an $\ell$-leaky forcing set of $G-e$, for by the inequality of the theorem, this is best possible. As such, this will yield $\Zl{G-e} = |B|$.

Fix $i'=1 \text{ or } 2$, $j'=1  \text{ or }  2$, let $i'' = 2  \text{ or }  1$ with $i'' \ne i'$ and $j'' = 2  \text{ or } 1$ with $j'' \ne j'$. 

If any leaf of $v_{i',j'}$ does not have a leak, then that leaf will immediately force $v_{i',j'}$. Hence, to prevent $v_{i',j'}$ to be forced immediately it is necessary to place leaks on each of the leaves of $v_{i',j'}$. However, if $v_{i',j''}$ is ever blue and any one of the $w_{i',k}$ does not have a leak, that $w_{i',k}$ can force $v_{i',j'}$.

It follows there two ways to prevent $v_{i',j'}$ from being forced immediately:

First, one could place all $\ell$ leaks on the leaves of $v_{i',1}$ and $v_{i',2}$, which will result in $v_{i',1}$ and $v_{i',2}$ left uncolored (as the $w_{i',k}$ cannot force), but all the other leaves will be able to force their neighbors resulting in only $v_{i',1}$ and $v_{i',2}$ uncolored. Since all $\ell$ leaks are already accounted for, $v_{i'',1}$ can force $v_{i',2}$ (as $e=v_{i',1}v_{i'',1}$ has been removed
). Thereafter, $v_{i'',2}$ can force $v_{i',1}$.

Second, one could place leaks on the leaves of $v_{i',j'}$ as well as all of the $w_{i',k}$. This requires at least $2 \lfloor \ell / 2 \rfloor$ leaks. Since preventing any other $v_{i,j}$ from being forced immediately requires at least $2 \lfloor \ell / 2 \rfloor \ge \ell-1$ leaks (or exactly $\ell$ if $\ell$ is even), all other of the $v_{i,j}$ will be forced by their leaves (or when $\ell = 3$, perhaps by a $w_{i,k}$ if the one extraneous leak on their only leaf). With at most one leak unaccounted for, one of $v_{i',j''}$ or $v_{i'',j''}$ can force $v_{i',j'}$.

Since $B$ is a forcing set and this is best possible, $\Zl{G-e} = |B|$. Altogether, we have $\Zl{G}-\Zl{G-e} = (|B|+2) - |B| = 2$.
\end{proof}

\usetikzlibrary{decorations.pathreplacing, positioning, calc}

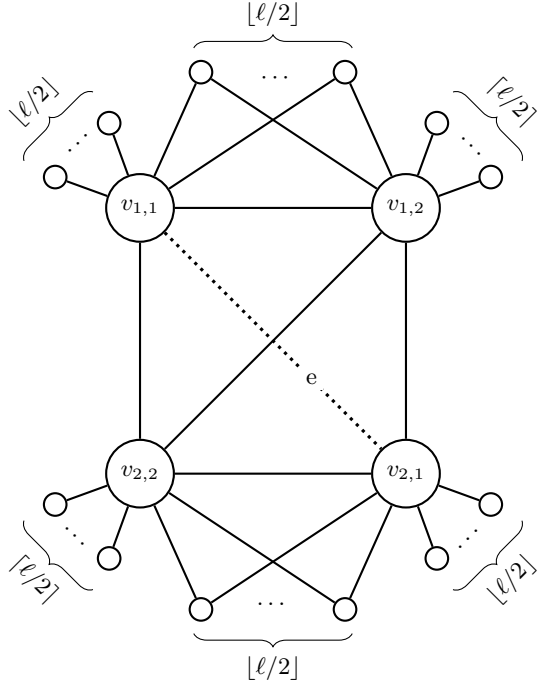
\begin{figure}
\centering
\begin{tikzpicture}[
    v/.style={circle, draw, thick, fill=white, inner sep=2pt, minimum size=9mm},
    leaf/.style={circle, draw, thick, inner sep=1pt, minimum size=3mm, fill=white},
    shared/.style={circle, draw, thick, inner sep=1pt, minimum size=3mm, fill=white},
    font=\small
]

\def\l{2}           
\def\R{2.5cm}       
\def\pairspread{45}  
\def\leafedgelen{1.2cm} 
\def\leafspread{25}  
\def\shareddist{1.2cm}
\def\sharedspread{15}  
\def\leafshift{0}   

\foreach \i in {1,2} {
    \pgfmathsetmacro{\pairCenterAngle}{90 - ((\i-1) * 360 / \l)}
    
    \pgfmathsetmacro{\angleA}{\pairCenterAngle - \pairspread}
    \pgfmathsetmacro{\angleB}{\pairCenterAngle + \pairspread}
    
    \node[v] (v\i1) at (\angleA:\R) {$v_{\i,2}$};
    \node[v] (v\i2) at (\angleB:\R) {$v_{\i,1}$};

    \pgfmathsetmacro{\leafAngleA}{\angleA - \leafshift}
    \node[leaf] (l\i1a) at ($(v\i1) + (\leafAngleA+\leafspread:\leafedgelen)$) {};
    \node[leaf] (l\i1b) at ($(v\i1) + (\leafAngleA-\leafspread:\leafedgelen)$) {};
    \coordinate (midL1) at ($(l\i1a)!0.5!(l\i1b)$);
    \node[rotate=\leafAngleA+90] at ($(v\i1)!1.1!(midL1)$) {$\dots$};
    \draw[thick] (v\i1) -- (l\i1a);
    \draw[thick] (v\i1) -- (l\i1b);
    
    \pgfmathsetmacro{\leafAngleB}{\angleB + \leafshift}
    \node[leaf] (l\i2a) at ($(v\i2) + (\leafAngleB+\leafspread:\leafedgelen)$) {};
    \node[leaf] (l\i2b) at ($(v\i2) + (\leafAngleB-\leafspread:\leafedgelen)$) {};
    \coordinate (midL2) at ($(l\i2a)!0.5!(l\i2b)$);
    \node[rotate=\leafAngleB+90] at ($(v\i2)!1.1!(midL2)$) {$\dots$};
    \draw[thick] (v\i2) -- (l\i2a);
    \draw[thick] (v\i2) -- (l\i2b);

    \node[shared] (s\i a) at (\pairCenterAngle+\sharedspread:\R+\shareddist) {};
    \node[rotate=\pairCenterAngle+90] at (\pairCenterAngle:\R+\shareddist-0.2cm) {$\dots$};
    \node[shared] (s\i b) at (\pairCenterAngle-\sharedspread:\R+\shareddist) {};
    
    \draw[thick] (v\i1) -- (s\i a); \draw[thick] (v\i2) -- (s\i a);
    \draw[thick] (v\i1) -- (s\i b); \draw[thick] (v\i2) -- (s\i b);

    \def\labelpos{above}
    \ifnum\i=2 \def\labelpos{below} \fi
    
    \coordinate (b\i s1) at ($(v\i1)!1.28!(l\i1b.center)$);
    \coordinate (b\i e1) at ($(v\i1)!1.28!(l\i1a.center)$);
    \draw[decorate, decoration={brace, amplitude=5pt, mirror, raise=4pt}] (b\i s1) -- (b\i e1)
        node[midway, sloped, \labelpos=9pt] {$\lceil \ell/2 \rceil$};
        
    \coordinate (b\i s2) at ($(v\i2)!1.28!(l\i2b.center)$);
    \coordinate (b\i e2) at ($(v\i2)!1.28!(l\i2a.center)$);
    \draw[decorate, decoration={brace, amplitude=5pt, mirror, raise=4pt}] (b\i s2) -- (b\i e2)
        node[midway, sloped, \labelpos=9pt] {$\lfloor \ell/2 \rfloor$};
        
    \coordinate (pair\i center) at ($(v\i1)!0.5!(v\i2)$);
    \coordinate (bs_\i s) at ($(pair\i center)!1.1!(s\i b.center)$);
    \coordinate (bs_\i e) at ($(pair\i center)!1.1!(s\i a.center)$);
    \draw[decorate, decoration={brace, amplitude=5pt, mirror, raise=4pt}] (bs_\i s) -- (bs_\i e)
        node[midway, sloped, \labelpos=9pt] {$\lfloor \ell/2 \rfloor$};
}

\draw[thick] (v11) -- (v12); 
\draw[thick] (v21) -- (v22); 
\draw[thick] (v12) -- (v21); 
\draw[thick] (v11) -- (v22); 
\draw[thick] (v11) -- (v21); 

\draw[very thick, dotted] (v12) -- (v22) node[midway, below right=10pt and 10pt, fill=white, inner sep=2pt] {e};

\end{tikzpicture}
\caption{An example of a graph with $\Zl{G} - \Zl{G-e}=2$ as given in part of the proof of Theorem \ref{thm:edgeremove}.}
\label{fig:gre}
\end{figure}

\begin{thm} \label{thm:vtxremove}
Let $G$ be a graph, let $\ell \ge 0$ and let $v$ be a vertex with degree $\deg(v)$, then \[-1 \leq \Zl{G-v}-\Zl{G}\leq \deg(v)\] and these bounds are tight when $\ell\geq 1$.
\end{thm}

\begin{proof}
    Let $S'$ be a minimum $\ell$-leaky forcing set of $G-v$. Then $S \cup \{v\}$ is an $\ell$- leaky forcing set of $G$ as all the same forces apply. Therefore, $\Zl{G} \leq \Zl{G-v}+1$ or $-1 \leq \Zl{G-v}-\Zl{G}$.

    Let $S$ be a minimum $\ell$-leaky forcing set of $G$. Then we claim that $S \cup N(v) \setminus \{v\}$ is an $\ell$-leaky forcing set of $G-v$. Specifically, we claim that the same sequence of forces for $S$ in $G$ can effectively apply to $S \cup N(v) \setminus \{v\}$ in $G-v$: Any force by $v$ has effectively already occurred as the to-be forced vertex is already blue. Additionally, 
    a neighbor of $v$ in $G$, requiring $v$ to be blue to force in $G$ can still force in $G-v$. Therefore, $S \cup N(v) \setminus \{v\}$ is an $\ell$-leaky forcing set of $G-v$ with size of $\Zl{G} + \deg(v) -1$ (if $v \in S$) or $\Zl{G} + \deg(v)$ (if $v \not\in S$). Hence, $\Zl{G-v} \le \Zl{G} + \deg(v)$ or, equivalently, $\Zl{G-v}-\Zl{G}\leq \deg(v)$.

   We will now demonstrate the bound is tight in both directions when $\ell \ge 1$. 

    Choose $\ell \ge 1$. Consider the graph $G = K_\ell +_v K_\ell$, that is two complete graphs with one shared vertex $v$. By Lemma \ref{lem:lowdegree}, every vertex except $v$ must be in a leaky forcing set. Since there are $2\ell$ vertices which could force $v$, $v$ will be forced regardless of the location of the leaks. So this is an $
    \ell$-leaky forcing set. Hence, $\Zl{K_\ell +_v K_\ell} = 2\ell -1$.
    Similarly, if we consider $K_\ell +_v K_\ell + w$, which is the graph $K_\ell +_v K_\ell$ with a leaf $w$ added to the shared vertex $v$. 
    By the same argument, $\Zl{K_\ell +_v K_\ell + w} = 2\ell.$ Hence, $\Zl{K_\ell +_v K_\ell}-\Zl{K_\ell +_v K_\ell + w}=-1$, so the lower bound is tight.
    
    Choose $\ell \ge 1$. Consider a rooted tree $T$ constructed as follows. Let $r$ be the root vertex with degree $d > \ell$; let $T$ have height 2; let and all non-root non-leaf vertices each have $\ell$ children. By Theorem \ref{thm:LeakyTree}, $\Zl{T}=d \ell$, the number of leaves of $T$. Consider $T-r$ which is $d$ disjoint trees each with $\ell$ leaves. However, the former child of $r$ of each tree now has degree $\ell$, so by Theorem $\ref{thm:LeakyTree}$, $\Zl{T-r} = d \ell+d$. Hence, $\Zl{T-r}-\Zl{T}=d$, so the upper bound is tight.
\end{proof}

\section{Graphs with Extreme 1-Leaky Forcing Number} \label{sec:extremal}
For any connected graph $G$ on the $n \ge 2$ vertices, we have $2 \le \Zone{G} \le n-1$. In this section, we characterize the connected graphs that achieve these extremal values, beginning with the case where \( \Zone{G} = 2 \). 

In the context of zero forcing, it is known that $\Zf{G} \leq 2$ if and only if $G$ is a graph on two parallel paths (see \cite{row2012technique}). However, the characterization differs significantly for the 1-leaky forcing number where only paths and cycles achieve $ \Zone{G} = 2$ in this case.

\begin{thm} \label{thm:Zoneclass2}
 Let $G$ be a connected graph on $n\geq 2$ vertices. Then $\Zone G = 2$ if and only if $G$ is a path graph or a cycle graph.
\end{thm}

\begin{proof}
The backwards direction holds as $\Zone{P_n} = \Zone{C_n} = 2$ (see Proposition \ref{prop:basicgraphresults}). It remains to show that if $G$ has $\Zone{G}=2$, then $G$ is $P_n$ or $C_n$

Suppose $\{x_1, y_1\}$ is a 1-leaky forcing set of $G$. Without loss of generality, suppose $y_1$ has the leak. It must be the case that $x_1$ is capable of initiating a forcing chain (i.e., a path $x_1, x_2, \ldots, x_{n-1}$ where $x_i$ forces $x_{i+1}$) that forces all of $G$ without $y_1$ forcing at all. 
Necessarily, since $x_i$ 
forces $x_{i+1}$, there cannot be any edge from $x_i$ to $x_k$ for $i+1<k<n-1$. Further, since the path $x_1, x_2, \ldots, x_{n-1}$ must encompass all vertices of $G$, except for $y_1$, the only edges among them must be of the form $x_i x_{i+1}$. It follows that the induced subgraph on $V(G) \setminus \{ y_1\}$ must be a path. Similarly, $V(G) \setminus \{x_1\}$ must be a path. If $y_1x_2$ is an edge, then a leak on $x_2$ will interrupt both forcing chains. If $y_1x_i$ is an edge for $2<i<n-1$, then the induced subgraph $V(G) \setminus \{x_1\}$ can't be a path. It follows that $y_1$ must be adjacent to $x_{n-1}$. The only remaining possibilities are whether or not $x_1$ and $y_1$ are adjacent; one case results in $C_n$ and the other results in $P_n$.
\end{proof}

Next, we will characterize connected graphs such that $\Zone{G} = n-1$. Note that for zero forcing, $Z(G)=n-1$ if and only if $G$ is the complete graph (see \cite{row2012technique}). In contrast, there are four such graphs in the $1$-leaky case. Let $K_n$ and $S_n$ be the complete graph and star graph on $n$ vertices, respectively. Define $K_{n-1}+v^\plainell$ to be the complete graph on $n-1$ vertices with an added leaf $v^\plainell$ and $K_n-e$ to be the complete graph on $n$ vertices with one edge removed. 

We will first prove some helpful lemmas. We start by providing a characterization for graphs with $\ell$-leaky forcing number equal to $n-1$ in terms of their $\ell$-leaky fort structure.

\begin{lemma}\label{lem:MinForts}
For a graph $G$ on $n$ vertices and $0\leq \ell\leq n$, $Z_{(\ell)}(G) \ge n-1$ if and only if the set of minimal $\ell$-leaky forts is precisely
    \begin{enumerate}
        \item singleton sets, one for each vertex of degree $\ell$ or less, and
        \item all pairs of vertices of degree at least $\ell+1$. 
    \end{enumerate}
\end{lemma}

\begin{proof}
    First, let us prove that if $\Zl{G} = n-1$ or $n$, then the minimal $\ell$-leaky forts are as described.

    If $\Zl{G} = n$, then every vertex must form its own singleton $\ell$-leaky fort. This only occurs if every vertex has degree $\ell$ or less.
    
    If $\Zl{G} = n-1$, then there are no $\ell$-leaky forcing sets of size $n-2$. Suppose for contradiction that there is a pair of vertices $u,v$ where $\{u,v\}$ is not an $\ell$-leaky fort (nor is $\{u\}$ and nor is $\{v\}$). Then, every $\ell$-leaky fort containing $u$ or $v$ contains a vertex in $V(G) \setminus \{u,v\}$. Hence, $V(G) \setminus \{u,v\}$ intersects every $\ell$-leaky fort and is an $\ell$-leaky forcing set by Proposition \ref{prop:LeakyIFFforts}. It follows that $\{u,v\}$, $\{u\}$, or $\{v\}$ must be a fort. By Proposition \ref{lem:lowdegree}, the latter two can only occur if either $u$ or $v$ have degree $\ell$ or less. Hence if $u$ and $v$ both have degree at least $\ell$ + 1, $\{u,v\}$ must be a fort.

    In reverse, suppose $\Zl{G} \le n-2$. Then by Proposition \ref{prop:LeakyIFFforts}, there is an $\ell$-leaky forcing set that excludes two vertices $x,y$. In which case, $\{x,y\}$ nor $\{x\}$ nor $\{y\}$ can be forts.
\end{proof}

Next, we show that $S_n$ is the only connected graph with high $1$-leaky forcing number and more than two leaves.

\begin{lemma}\label{lem:StarLeaf}
If $G$ is a connected graph on $n$ vertices with at least two leaves, then $\Zone{G} = n-1$ if and only if $G\cong S_n$.
\end{lemma}

\begin{proof}
If $G\cong S_n$, then it is clear that $\Zone{G} = n-1$. 

Now, assume $G$ is a connected graph on n vertices with at least two leaves such that $\Zone{G} = n-1$ and $G\ncong S_n$. Suppose $G$ has two vertices $u_1$ and $u_2$ with neighbors $v_1$ and $v_2$, respectively, such that $\deg(u_1)>2$, $\deg(u_2)>2$, and $\deg(v_1)=\deg(v_2)=1$. Let $S=V(G)\setminus \{u_1,u_2\}$ be a set of initially blue vertices. Then, at least one vertex $u_i$ will be forced by $v_i$ for $i=1,2$, since at most one of $\{v_1,v_2\}$ can have a leak. The remaining white vertex can be forced by its non-leaf neighbor, so $S$ is a $1$-leaky forcing set and $\Zone{G} \leq n-2$.

Otherwise, $G$ has a vertex $u$ with neighbors $v_1$, $v_2$, and $w$ such that $\deg(u)>3$ and $\deg(v_1)=\deg(v_2)=1$. Let $S=V(G) \setminus \{u,w\}$ be a set of initially blue vertices. Then, $u$ will be forced by $v_i$ for $i=1,2$, since at most one of $\{v_1,v_2\}$ can have a leak. Finally, $u$ can force $w$, so $S$ is a $1$-leaky forcing set and $\Zone{G} \leq n-2$.
\end{proof} 

We now state and prove our characterization of graphs whose $1$-leaky forcing number is $n-1$.

\begin{thm} \label{thm:Zoneclass}
    Let $G$ be a connected graph on $n\geq 3$ vertices. Then $\Zone G = n-1$ if and only if $G\in\{K_n,S_n,K_{n-1}+v^\plainell,K_n -e\}$ where $v^\plainell$ is an added leaf vertex.
\end{thm}

\begin{proof}

Let $n=3$. Then, $G\in\{K_3,K_3-e\}$. If $G\cong K_3$, $Z_{(1)}(K_3)=2$ by Proposition \ref{prop:basicgraphresults}. If $G\cong K_3-e\cong P_3$,  $\Zone{G}=2$ by Proposition \ref{prop:basicgraphresults}. Since these are the only two connected graphs for $n=3$, the result is established for this case.

Let $n=4$. The only connected graphs on 4 vertices are $\{K_4,K_4-e,K_3+v^\plainell, S_4, C_4, P_4\}$. So, consider $G \in \{K_4,K_4-e,K_3+v^\plainell, S_4, C_4, P_4\}$. If $G\cong C_4$ or $G\cong P_4$, then $\Zone{G}= 2$ by Proposition \ref{prop:basicgraphresults}. Also, if $G\cong K_4$, $\Zone{G} =  3$ by Proposition \ref{prop:basicgraphresults}. So, it remains to establish the result for  $G\in\{K_4-e,K_3+v^\plainell, S_4\}$. 

If $G\cong K_4-e$, let $V(G)=\{v_1,v_2,v_3,v_4\}$ such that $e=v_1v_2$. One can argue that no subset of vertices of size two are $1$-leaky sets, and thus $3 \le \Zone{G}$ and equality follows noting the subset $\{v_1,v_2,v_3\}$ of vertices is a $1$-leaky forcing set. 

If $G\cong K_3+v^\plainell$, let $u\in V(G)$ be the vertex adjacent to the leaf $v^\plainell$ and let $\{v_1,v_2\}$ be the remaining vertices. Every $1$-leaky forcing set must contain the leaf $v^\plainell$. The set $\{u,v^\plainell\}$ is not a zero forcing set since $u$ has two white neighbors and $v^\plainell$ has no white neighbors. The set $\{v_i,v^\plainell\}$ for $i=1,2$ is not a $1$-leaky forcing set since $v_i$ has two white neighbors and placing the leak on $v^\plainell$ prevents it from forcing. Therefore, $\Zone{G}\geq 3$. The set $\{v_1,v_2,v^\plainell\}$ is a $1$-leaky forcing set of size $3$, since any vertex in the set can force $u$. Thus, $\Zone{G}= 3$.

Finally, if $G\cong S_4$, every $1$-leaky forcing set contains the set of $3$ leaves. Furthermore, the set of leaves is a $1$-leaky forcing set since the center vertex can always be forced by a leaf without a leak. Therefore $\Zone{G}=3$.

Now suppose $n\geq 5$. If $G$ has more than one leaf, then $\Zone{G} = n-1$ if and only if $G\cong S_n$ by Lemma \ref{lem:StarLeaf}. So we now assume that $G$ has at most one leaf.

First, we suppose $G\in\{K_n,K_{n-1}+v^\plainell,K_n -e\}$ and show that $\Zone{G} = n-1$. If $G\cong K_n$, then $\Zone{G} =n-1$ by Proposition \ref{prop:KP}. Next, suppose $G\cong K_{n-1}+v^\plainell$ and let $H$ be the subgraph of $G$ which is isomorphic to $K_{n-1}$. By Lemma \ref{lem:lowdegree}, $\{v^\plainell\}$ is a minimal $1$-leaky fort. No other minimal $1$-leaky fort will contain this vertex, so we turn our attention to the vertices of $H$. No single vertex $x$ in $H$ can be a $1$-leaky fort because every other vertex in $H$ is adjacent to $x$. Therefore, a minimal $1$-leaky fort in $H$ must have size at least two. Let $x$ and $y$ be any two vertices in $V(H)$. Every vertex in $V(H)\setminus\{x,y\}$ is adjacent to both $x$ and $y$ and if $v^\plainell$ is adjacent to $x$ or $y$, it will be the only vertex outside the $1$-leaky fort which is adjacent to exactly one vertex in the fort. Therefore, $\{x,y\}$ is a minimal $1$-leaky fort and by Lemma \ref{lem:MinForts}, $\Zone{G} = n-1$.

Finally, suppose $G\cong K_{n}-e$. No singleton set $\{x\}$ in $G$ can be $1$-leaky fort because all but one other vertex in $G$ is adjacent to $x$ (and since $n\geq 5$, $x$ has at least three neighbors). Therefore, a minimal $1$-leaky fort must have size at least two. Let $x$ and $y$ be any two vertices in $V(G)$. If $x$ and $y$ are both incident to $e$ or neither are incident to $e$, they have at least three common neighbors in $V(G)\setminus\{x,y\}$. If exactly one of $x$ or $y$ is incident to $e$, they have at least two common neighbors. In both cases, $\{x,y\}$ is a minimal $1$-leaky fort and by Lemma \ref{lem:MinForts}, $\Zone{G} = n-1$.

Now suppose that $G\not\in \{K_n,K_{n-1}+v^\plainell,K_n -e\}$. Since $G\not\cong  K_n$ and $G\not\cong K_n-e$, we know the graph must contain at least two non-edges (i.e. a pair of vertices which is non-adjacent). We will consider cases based on the number of leaves that $G$ contains.

\textbf{Case 1:} Suppose $G$ has no leaves, so the minimum degree of $G$ is at least 2. 

\textbf{Case 1.1:} The minimum degree of $G$ is $n-2$. 

Every vertex in $G$ is incident to at most one non-edge. In particular, since $G$ contains at least two non-edges, there exists vertices $x,y,w,z\in V(G)$ such that $xy\not\in E(G)$ and $wz\not\in E(G)$. Due to the minimum degree, all other edges must exist between pairs of these vertices. Since $n\geq 5$, there exists another vertex $v$ and because the minimum degree is $n-2$, $x,y,w,$ and $z$ must be adjacent to $v$. We will now demonstrate a 1-leaky forcing set of size $n-2$. Let $S=V(G)\setminus \{x,w\}$ be a set of initially blue vertices. If neither $y$ nor $z$ has the leak, then they force $w$ and $x$ respectively. Otherwise, without loss of generality, if $y$ has the leak, then $z$ forces $x$ and $x$ forces $w$. Therefore, we have found a 1-leaky forcing set of size $n-2$, so $\Zone{G} \leq n-2$.

\textbf{Case 1.2:} The minimum degree of $G$ is at most $n-3$. 

There exists a vertex $x\in V(G)$ which is incident to at least two non-edges. Let two of the vertices in $G$ which are not adjacent to $x$ be $y$ and $z$. Since the minimum degree is at least two, $x$ has at least two neighbors in $V(G)\setminus \{x,y,z\}$, and $y$ and $z$ each have at least one neighbor in $V(G)\setminus \{x,y,z\}$ (as they could be adjacent to each other).

Suppose that $y$ and $z$ share a common neighbor and call this neighbor $w$. Let $S=V(G)\setminus \{x,w\}$ be a set of $n-2$ initially blue vertices. Since $y$ and $z$ are both adjacent to $w$ but not to $x$, either of them can force $w$. Then, any neighbor of $x$ can force $x$ (there are at least two). Thus, $S$ is a zero forcing set for $G$, and since each force could be accomplished in at least two distinct ways, $S$ is a 1-leaky forcing set by Lemma \ref{lem:twodifferentleaky}. Thus, $\Zone{G} \leq n-2$.

Otherwise, suppose that $y$ and $z$ do not share a common neighbor. Let $w$ be a neighbor of $y$ which is not $z$ and let $v$ be a neighbor of $z$ which is not $y$. Let $S=V(G)\setminus \{y,z\}$ be a set of $n-2$ initially blue vertices. If neither $w$ nor $v$ is leaky, then they force $y$ and $z$ respectively. If, without loss of generality, $w$ is leaky, then $v$ forces $z$. The second neighbor of $y$ (which could be z) can then force $y$. Therefore, $S$ is a 1-leaky forcing set and $\Zone{G} \leq n-2$.

\textbf{Case 2:}  Suppose $G$ has exactly one leaf, denoted $w^*$.

By Lemma \ref{lem:lowdegree} leaf $w^*$ must be in every 1-leaky forcing set. Note that $G - w^*$ has $n-1$ vertices and no leaves. Therefore, if $G - w^*$ is a graph which is not $K_{n-1}$, $K_{n-2}+v^*$, or $K_{n-1}-e$ then by Case 1, $\Zone{G-w^*} \leq n-3$ and so $\Zone{G} \leq n-2$.

We now consider the case where $G-w^*$ is one of the graphs $K_{n-1}$, $K_{n-2}+v^*$, or $K_{n-1}-e$. First, note that $G-w^*$ cannot be $K_{n-1}$ since we assumed $G$ is not $K_{n-1}+v^*$.

\textbf{Case 2.1:} $G-w^*\cong K_{n-2}+v^\plainell$.

It must be that $G$ is $K_{n-2}$ with a $P_2$ appended to a vertex (since otherwise $G$ would contain more than one leaf). In this case, let $x$ be the vertex adjacent $v^*$, and let $y$ be the vertex in $K_{n-2}$ which is adjacent to $x$. Let $S=V(G)\setminus\{x,y\}$ be a set of $n-2$ initially blue vertices. Since $n\geq 5$, vertex $y$ has at least two neighbors in $K_{n-2}$ and can be forced by one of these neighbors. Finally, $x$ can be forced by $y$ or $v^\plainell$. Thus, $S$ is a zero forcing set for $G$, and since each force could be accomplished in at least two distinct ways, $S$ is a 1-leaky forcing set by Lemma \ref{lem:twodifferentleaky}. Thus, $\Zone{G} \leq n-2$.

\textbf{Case 2.2:} $G-w^*\cong K_{n-1}-e$

Then $G$ is $K_{n-1}-e$ with the leaf $w^*$ appended to some vertex. Let $y$ and $z$ be the vertices incident to the missing edge $e$.

First, suppose $w^*$ is adjacent to a vertex of degree $n-2$ in $K_{n-1}$, denoted $x$, and let $S=V(G)\setminus\{x,y\}$ be a set of $n-2$ initially blue vertices. Observe that $z$ or $w^*$ can force $x$ and then $y$ can be forced by any of its neighbors (since $n\geq 5$, $y$ has at least two neighbors). Thus, $S$ is a zero forcing set for $G$, and since each force could be accomplished in at least two distinct ways, $S$ is a 1-leaky forcing set by Lemma \ref{lem:twodifferentleaky}. Thus, $\Zone{G} \leq n-2$.

Otherwise, suppose without loss of generality that $w^*$ is adjacent to $y$. Let $S=V(G)\setminus\{x,y\}$ be a set of $n-2$ initially blue vertices. If $z$ is not leaky, then $z$ can force $x$ and $x$ can force $y$. Otherwise, if $z$ is leaky, $w^*$ can force $y$ and $y$ can force $x$. Therefore $S$ is a $1$-leaky forcing and $\Zone{G} \leq n-2$.
\end{proof}


\bibliographystyle{acm}
\bibliography{leaky_bib}

\end{document}